\documentclass[reqno]{amsart}

\usepackage{xypic,hyperref,todonotes,comment}
\usepackage{graphics,amssymb,mathtools}
\usepackage{pagecolor,lipsum}
\usepackage{latexsym}
\usepackage{mathrsfs}
\xyoption{curve}
\usepackage{hyperref}
\hypersetup{ 
colorlinks,
citecolor=cyan,
filecolor=blue,
linkcolor=black,
urlcolor=blue
}

\newtheorem{lemma}{Lemma}[section]
\newtheorem{proposition}[lemma]{Proposition}
\newtheorem{theorem}[lemma]{Theorem}
\newtheorem{corollary}[lemma]{Corollary}

\newtheorem*{theoremI}{Theorem A}
\newtheorem*{theoremII}{Theorem B}

\theoremstyle{definition}

\newtheorem{definition}[lemma]{Definition}
\newtheorem{remark}[lemma]{Remark}

\newcommand{\mbb}[1]{\mathbb{#1}}
\newcommand{\mbf}[1]{\mathbf{#1}}

\newcommand{\C}{\mathbb{C}}
\newcommand{\ot}{\otimes}
\newcommand{\Hom}{\mathrm{Hom}}
\newcommand{\Ext}{\mathrm{Ext}}

\newcommand{\Aut}{\mathrm{Aut}}
\newcommand{\rep}{\mathrm{rep}}
\newcommand{\bt}{\text{\tiny$\bullet$}}

\newcommand{\h}{\mathfrak{h}}
\newcommand{\g}{\mathfrak{g}}
\newcommand{\n}{\mathfrak{n}}
\newcommand{\U}{\mathscr{U}}
\renewcommand{\b}{\mathfrak{b}}
\renewcommand{\hat}{\widehat}
\renewcommand{\O}{\mathscr{O}}

\newcommand{\qbinom}[2]{\left[\begin{array}{c} #1\\ #2\end{array}\right]}

\title[]{Twists of quantum Borel algebras}
\date{\today}

\author{Cris Negron}
\thanks{This work was supported by NSF Postdoctoral Research Fellowship DMS-1503147}
\email{negronc@mit.edu}
\address{Department of Mathematics, Massachusetts Institute of Technology, Cambridge, MA 02139}

\begin{document}
\maketitle

\begin{abstract}
We classify Drinfeld twists for the quantum Borel subalgebra $u_q(\b)$ in the Frobenius-Lusztig kernel $u_q(\g)$, where $\g$ is a simple Lie algebra over $\C$ and $q$ an odd root of unity.  More specifically, we show that alternating forms on the character group of the group of grouplikes for $u_q(\b)$ generate all twists for $u_q(\b)$, under a certain algebraic group action.  This implies a simple classification of Hopf algebras whose categories of representations are tensor equivalent to that of $u_q(\b)$.  We also show that cocycle twists for the corresponding De Concini-Kac algebra are in bijection with alternating forms on the aforementioned character group.
\end{abstract}

\section{Introduction}

In this paper we classify Drinfeld twists for the quantum Borel subalgebra $u_q(\b)$ in the Frobenius-Lusztig kernel $u_q(\g)$, for a simple Lie algebra $\g$ over $\C$ at an odd root of unity $q$ (see Section~\ref{sect:uq}).  The algebra $u_q(\g)$ is also known as the {\it small quantum group}.
\par

Recall that Drinfeld twists for a given Hopf algebra $H$ correspond to tensor structures on the forgetful functor from $\rep(H)$ to $Vect$.  Here $\rep(H)$ denotes the tensor category of finite dimensional $H$-modules.  Two twists are said to be gauge equivalent if their corresponding functors are naturally isomorphic.  We let $\mathrm{Tw}(H)$ denote the set of gauge equivalence classes of twists for $H$.  There is a group of {\it twisted autoequivalences} of $H$ which acts on $\mathrm{Tw}(H)$, and the resulting quotient parametrizes isomorphism classes of Hopf algebras $K$ which admit a tensor equivalence $\rep(K)\overset{\sim}\to \rep(H)$.  (See Sections~\ref{sect:twists} and~\ref{sect:twAuts}.)
\par

For the small quantum Borel, it is known that the unipotent algebraic group $\mathbb{U}$, corresponding to the nilpotent subalgebra $\n=[\b,\b]\subset\b$, acts on the collection of twists for $u_q(\b)$ by way of twisted automorphisms~\cite{arkhipovgaitsgory03,uqJ}.  Basic considerations also establish an embedding $\mathrm{Alt}(G^\vee)\to \mathrm{Tw}(u_q(\b))$, where $\mathrm{Alt}(G^\vee)$ denotes the set of alternating bilinear forms on the character group $G^\vee$ of the Cartan subgroup $G=G(u_q(\b))$.  We show below that the set of alternating forms on $G^\vee$ generates {\it all} twists for $u_q(\b)$ under the aforementioned action of $\mathbb{U}$.

\begin{theoremI}[\ref{thm:unip}]
There is an equality $\mathrm{Tw}\left(u_q(\b)\right)=\mathbb{U}\cdot \mathrm{Alt}(G^\vee)$.
\end{theoremI}

We also show in Proposition~\ref{prop:triv1} that the De Concini-Kac algebra $U^{DK}_q(\b)$ admits no {\it cocycle} twists up to gauge equivalence, save for those coming from the group of grouplikes.  A version of Theorem A can also be shown to hold for quantum Kac-Moody algebras.  In this case one should complete $u_q(\b)$ relative to its grading by the root lattice.
\par

As a consequence of Theorem A we derive the following classification result for Hopf algebras which are tensor equivalent to a small quantum Borel algebra.  The classification can alternately be obtained from works of Andruskiewitsch, Angiono, Iglesias, and Schneider~\cite{ASIII,aai17,angionoiglesias}, whom we refer to collectively as AAIS.

\begin{theoremII}[\ref{thm:equiv}, AAIS]
Suppose $H$ is a Hopf algebra such that $\rep(H)$ is equivalent to $\rep(u_q(\b))$ as a tensor category.  Then there is an alternating form $B$ on $G^\vee$ so that $H\cong u_q(\b)^B$ as Hopf algebras.  When $\mathrm{ord}(q)$ is coprime to the determinant of the Cartan matrix for $\g$, the form $B$ is uniquely determined by $H$.
\end{theoremII}

Here $u_q(\b)^B$ denotes the Hopf algebra twisted by $B$ (see Section~\ref{sect:twists}).  We note that Theorem B is well-known to hold in the case $\g=\mathfrak{sl}_2(\C)$.  In this case the quantum Borel is a Taft algebra $T_q$, there are no nontrivial alternating forms on the group of grouplikes, and hence any Drinfeld twist $T_q^J$ will be Hopf isomorphic to $T_q$ (see also \cite{mombelli10,etingofostrik04}).
\par

Let us say a few words regarding the cited works~\cite{ASIII,aai17,angionoiglesias}, which concern {\it liftings} of pointed, coradically graded, Hopf algebras (among other things).  A solution to the lifting problem, in our particular case, implies a direct classification all cocycle twisted algebras $(u_q(\b)^\ast)_J$, up to Hopf isomorphism~\cite{grunenfeldermastnak,angionoiglesias}.  The approaches taken in~\cite{ASIII,aai17,angionoiglesias} involve an analysis of certain Hopf deformations of a given cordially graded algebra, rather than a classification of the twists $J$ themselves.  Hence the work of AAIS offers an alternate means of arriving at Theorem B.
\par

The present work is motivated primarily by a conjecture regarding Drinfeld twists for the small quantum group $u_q(\g)$.  Recall that $u_q(\g)$ is generated by the positive and negative Borel subalgebras $u_q(\b_{\pm})\subset u_q(\g)$.  We ask in~\cite{den,uqJ} if there is an equality
\begin{equation}\label{eq:99}
\mathrm{Tw}(u_q(\g))=\mathbb{G}\cdot \mathrm{BD}(u_q(\g)),
\end{equation}
where $\mathrm{BD}(u_q(\g))$ denotes the collection of {\it Belavin-Drinfeld twists} for $u_q(\g)$, and $\mathbb{G}$ is the simply-connected semisimple algebraic group with Lie algebra $\g$.  As with $\mathbb{U}$ above, $\mathbb{G}$ acts by twisted automorphisms on $\mathrm{Tw}(u_q(\g))$.
\par

Belavin-Drinfeld twists for $u_q(\g)$ can be described vaguely as follows: Each twist in $\mathrm{BD}(u_q(\g))$ is specified by a choice of combinatorial data on the Dynkin diagram for $\g$.  This data is called a {\it Belavin-Drinfeld triple}~\cite{belavindrinfeld}.  Such a choice of data specifies an interaction between the positive and negative roots for $\g$, and hence an interaction between $u_q(\b_+)$ and $u_q(\b_-)$ in $u_q(\g)$.  We use the prescribed interaction between the positive and negative Borels, along with the $R$-matrix for $u_q(\g)$, to produce an especially non-trivial twist for $u_q(\g)$.  For example, one can drastically change the group of grouplikes after applying such a twist to $u_q(\g)$ (see~\cite{uqJ}), and in an extreme {\it dynamical} case one can make $u_q(\g)$ self-dual after twisting (see~\cite{etingofnikshych01}).  
\par

Hence, if we consider only the positive or negative Borel $u_q(\b)$ in $u_q(\g)$, we expect that there should be no such interesting twists.  Specifically, we expect from the conjectural equality~\eqref{eq:99} that all twists for $u_q(\b)$ should be induced by twists on the group of grouplikes, up to gauge equivalence and Hopf isomorphism.  Theorems A and B verify that this is in fact the case.

\subsection*{A description of the proof of Theorem A}

As our general methods may be of interest to readers, we describe here the proof of Theorem A.
\par

Take $u_q=u_q(\b)$.  Below we make heavy use of a natural grading on $u_q$ by the root lattice for $\g$, which we denote by $Q$.  Let $Q^+$ denote the $\mathbb{Z}_{\geq 0}$-span of the positive roots, minus $0$.  Given a twist $J$ for $u_q$ we can write $J=B+\sum_{\mu\in Q^+}J_\mu$, where each $J_\mu$ of $Q$-degree $\mu$ and $B$ a twist for $\C[G]$.  In general, one would like to eliminate the terms $J_\mu$ in some orderly fashion to produce an equivalence between $J$ and the corresponding degree $0$ term $B$.
\par

We show that the nonvanishing terms $J_\mu$ appearing in $J$ specify degree $2$ cohomology classes $[J_\mu]\in H^2((u_q^\ast)_B)$, where $(u^\ast_q)_B$ is the cocycle twist of the dual $u^\ast_q$ by the degree $0$ term $B$.  Furthermore, the term $J_\mu$ can be eliminated via a gauge transformation in $u_q$ if and only if the corresponding cohomology class $[J_\mu]$ vanishes (see Section~\ref{sect:506}).
\par

As a consequence of a natural identification between $H^2((u_q^\ast)_B)$ and the nilpotent subalgebra $\n$ in $\b$ (Proposition~\ref{prop:HuB}), we find that there are only finitely many problem terms $J_\gamma$ in $J$ which {\it cannot} be eliminated via gauge transformation in $u_q$, up to scaling.  Namely, there is one such class for each positive root of $\g$.
\par

One shows that these problem terms {\it can} in fact be eliminated via gauge transformation by a distinguished collection of elements $\{\mbf{v}_\gamma(\lambda)\}_{\gamma\in \Phi^+,\lambda\in \mbb{C}}$ in a completion $\hat{U}_q(\b)$ of Lusztig's big quantum Borel algebra.  One shows further that the collective gauge actions of the $\mathbf{v}_\gamma(\lambda)$ reproduce a known action of $\mathbb{U}$ on $\mathrm{Tw}(u_q)$ (Section~\ref{sect:twAuts},~\cite{arkhipovgaitsgory03}).  It follows that any twist $J$ for $u_q$ can be reduced to the corresponding grouplike term $B$ via gauging in $u_q$ and the action of $\mbb{U}$.
\par

A formalization of the above information, and reference to a known equality $\mathrm{Tw}(\C[G])=\mathrm{Alt}(G^\vee)$~\cite{gelaki02,guillotkassel09}, produces the identification $\mathrm{Tw}(u_q)=\mathbb{U}\cdot \mathrm{Alt}(G^\vee)$.

\subsection*{Outline}
Sections~\ref{sect:uq} and~\ref{sect:twists} provide background information on quantum Borel algebras and twists.  Section~\ref{sect:Uq} provides a duality between Lusztig's divided power algebra, and the corresponding De Concini-Kac algebra.  Section~\ref{sect:drinfH} establishes an essential relationship between twists for quantum Borel algebras and cohomology classes for the dual algebra.  Section~\ref{sect:cohom} is dedicated to a presentation of the cohomologies of the relevant dual algebras.  In Section~\ref{sect:classTw} we prove a primordial version of Theorems A and B.  In Section~\ref{sect:proofs} we recall some information from~\cite{arkhipovgaitsgory03,uqJ}, and prove Theorems A and B.  In Section~\ref{sect:UDK} we provide proofs of some technical results from Section~\ref{sect:cohom}.

\subsection*{Acknowledgements}

Thanks to Pavel Etingof and Dmitri Nikshych for helpful conversations.  Thanks also to Agust\'{i}n Garc\'{i}a for providing some information regarding lifting problems for pointed Hopf algebras.

\section{Quantum Borel algebras}
\label{sect:uq}

Given any simple Lie algebra $\g$ we will consider a root of unity $q$ of a fixed odd order $l$.  In type $G_2$ we also require $l$ to be coprime to $3$, in accordance with~\cite{lusztig90II}.  In specific types we make the following additional assumptions on the order of $q$:
\begin{itemize}
\item In the simply laced case $l\neq 3$.
\item In types $B_n$, $C_n$, and $F_4$, $l\neq 3,5$.
\item In type $G_2$, $l\neq 7$.
\end{itemize}
We make these assumptions so that certain cohomology groups will vanish (see the proofs of Propositions~\ref{prop:HUB} and~\ref{prop:HuB}).

\subsection{Small quantum Borel algebras}

Let $\g$ be a simple Lie algebra with fixed Cartan $\h$, Dynkin diagram $\Gamma$, and $\b=\b_+\subset \g$ the positive Borel subalgebra.  We let $\Phi$ denote the collection of roots for $\g$, and $(?,?):\h\times\h\to \C$ denote the unique scaling of the Killing form so that $(\alpha,\alpha)=2$ whenever $\alpha$ is a short simple root.
\par

If we take $D$ to be the ratio of the lengths of a long root in $\Gamma$ to that of a short root, then for the Cartan integers we have
\[
\langle \alpha,\beta\rangle=D\langle\beta,\alpha\rangle,\ \ \text{whenever $\beta$ is long and $\alpha$ is short}. 
\]
In this case $(\alpha,\beta)$ is equal to the Cartan integer $\langle \alpha,\beta\rangle$.  The number $D$ is $1,2$, or $3$ and the matrix $\left[(\alpha,\beta)\right]_{\alpha,\beta\in\Gamma}$ representing the Killing form is integer valued with nonpositive entries off the diagonal.
\par

For $\alpha$ a simple root we define
\[
q_{\alpha}=q^{(\alpha,\alpha)/2}=\left\{\begin{array}{ll}
q^{D} &\ \mathrm{when\ }\alpha\mathrm{\ is\ a\ long},\\
q &\ \mathrm{when\ }\alpha\mathrm{\ is\ a\ short.}
\end{array}\right.
\]
We also adopt the standard expressions
\[
[n]_{q_\alpha}=(q_\alpha^n-q_\alpha^{-n})/(q_\alpha-q_\alpha^{-1})\ \ \mathrm{and}\ \ \qbinom{n}{m}_{q_\alpha}=\frac{[n]_{q_\alpha}\dots [n-m+1]_{q_\alpha}}{[m]_{q_\alpha}\dots [1]_{q_\alpha}}. 
\]
The small quantum Borel algebra is given by generators and relations as
\[
u_q(\b)=\C\langle E_\alpha, K_\alpha^{\pm 1}:\alpha\in \Gamma\rangle/(\mathrm{relations}),
\]
where the set of relations are
\[
[K_{\alpha},K_{\beta}]=0,\ \  K_{\alpha}E_\beta=q^{(\alpha,\beta)}E_\beta K_\alpha,
\]
\begin{equation}\label{eq:149}
\sum_{i=0}^{1-\langle\alpha,\beta\rangle}(-1)^i\qbinom{1-\langle\alpha,\beta\rangle}{i}_{q_\alpha}E_\alpha^{1-\langle\alpha,\beta\rangle-i}E_\beta E_\alpha^i=0,
\end{equation}
\[
K_\alpha^l=1,\ \  E_\mu^l=0\ \ \forall\ \mu\in \Phi^+.
\]
The $E_\mu$ here are given by an action of the Braid group as in~\cite{lusztig90,lusztig90II,deconcinikac89}.  The relations~\eqref{eq:149} are called the (quantum) {\it Serre relations}.
\par

The Hopf structure on $u_q(\b)$ is given by $\epsilon(E_\alpha)=0$, $\epsilon(K_\alpha)=1$, and
\[
\Delta(K_{\alpha})=K_\alpha\ot K_\alpha,\ \ \Delta(E_\alpha)=E_\alpha\ot 1+ K_\alpha\ot E_\alpha,
\]
\[
S(K_\alpha)=K_{\alpha}^{-1},\ \ S(E_\alpha)=-K^{-1}_\alpha E_\alpha.
\]
We let $G$ denote the group of grouplikes in $u_q(\b)$, which happens to be the subgroup generated by the elements $K_\alpha$.  We have $G\cong (\mathbb{Z}/l\mathbb{Z})^{\mathrm{rank}(\g)}$ with free basis over $\mathbb{Z}/l\mathbb{Z}$ given by the $K_\alpha$, $\alpha\in \Gamma$.

\subsection{Lusztig's divided powers algebra}
\label{sect:Lus}

Let $U_q(\b)$ denote Lusztig's quantum Borel algebra~\cite{lusztig90,lusztig90II}.  We will not give a detailed presentation of $U_q(\b)$ here, but remark on a few specific points.  The Hopf algebra $U_q(\b)$ is generated by elements $E_\alpha$ and $K_\alpha$, for $\alpha$ simple roots, and divided powers $E_\alpha^{(l)}$.  There is a Hopf inclusion $u_q(\b)\to U_q(\b)$ which sends the generators for $u_q(\b)$ to the corresponding generators in $U_q(\b)$.  The inclusion $u_q(\b)\to U_q(\b)$ identifies $u_q(\b)$ with a normal Hopf subalgebra in $U_q(\b)$, and it fits into a short exact sequence of Hopf algebras
\[
1\to u_q(\b)\to U_q(\b)\to U(\n)\to 1,
\]
where $U(\n)$ is the usual universal enveloping algebra.  The final map $U_q(\b)\to U(\n)$ sends the divided powers $E_\alpha^{(l)}$ to $e_\alpha$ in $\n=\oplus_{\mu\in \Phi^+}\C e_\alpha$~\cite[Lem. 8.9]{lusztig90II}.  By an {\it exact sequence} we mean that $U_q(\b)$ is faithfully flat over $u_q(\b)$, that the kernel of $U_q(\b)\to U(\n)$ is generated by augmentation ideal for $u_q(\b)$, and that $u_q(\b)$ is the $U(\n)$-coinvariants in $U_q(\b)$, as in~\cite{arkhipovgaitsgory03} (see also~\cite[Thm. 3.5.3, Prop. 3.4.3]{montgomery}).
\par

The divided power operators $E_\alpha^{(l)}$ are connected with elements in $u_q(\b)$ via the comultiplication
\[
\Delta(E_\alpha^{(l)})=\sum_{0\leq i\leq l} q^{-i(l-i)}K_\alpha^iE_\alpha^{(l-i)}\ot E_\alpha^{(i)}
\]
~\cite[Prop. 4.8]{lusztig90}.

\subsection{The root lattice and root lattice gradings}

We let $Q=Q(\g)$ denote the root lattice for $\g$, and $Q^+$ denote the positive root lattice
\[
Q^+=\left\{\sum_{\alpha\in\Gamma}n_\alpha\alpha:n_\alpha\geq 0\text{ for all }\alpha\text{ and }n_\beta>0\text{ for some }\beta\right\}.
\]
\par

The quotient $Q/lQ$ is identified with $G$ in the obvious manner, $\alpha\mapsto K_\alpha$.  The Killing form on $\h^\ast$ induces a form on $Q/lQ$, and hence on $G$.  Whence we have a canonical map from $Q/lQ\cong G$ to the character group $G^\vee$ which sends $\alpha\in Q/lQ$, or $K_\alpha\in G$, to the function $(K_\beta\mapsto q^{(\alpha,\beta)})$.  We call this map the {\it Killing map}, and denote it $\kappa:G\to G^\vee$.  In this way we view roots for $\g$ as characters on $G$.  Note that the Killing map will be an isomorphism exactly when the  Killing form on $G$ is non-degenerate.  This in turn occurs exactly when the determinant of the Cartan matrix is invertible in $\mathbb{Z}/l\mathbb{Z}$.
\par

The coradical gradings on $u_q(\b)$ and $U_q(\b)$ can be seen as induced by a natural gradings by the root lattice
\[
\deg_Q(K_\alpha)=0,\ \deg_Q(E_\alpha)=\alpha,\ \deg_Q(E^{(l)}_\alpha)=l\alpha.
\]
We have the obvious group map
\[
|?|:Q\to \mathbb{Z},\ \ |\sum_{\alpha\in \Gamma} n_\alpha \alpha|=\sum_{\alpha\in \Gamma}n_\alpha,
\]
and see that the coradical gradings on $u_q(\b)$ and $U_q(\b)$ agree with the $\mathbb{Z}$-grading induced by the above map.
\par

We adopt the following partial ordering on $Q$: For elements $\eta=\sum_\alpha n_\alpha\alpha$ and $\eta'=\sum_\alpha n'_\alpha\alpha$ in the root lattice we write $\eta\leq \eta'$ if $n_\alpha\leq n'_\alpha$ for each $\alpha$.  We write $\eta<\eta'$ if $\eta\leq \eta'$ and there exists a simple $\beta$ with $n_\beta<n'_\beta$.  Note that the above map $|?|:Q\to \mathbb{Z}$ preserves our partial ordering.

\section{Basic information for twists}
\label{sect:twists}

We cover here some basic information on twists, and recall a classification of twists for the group algebra of a finite abelian group, from~\cite{gelaki02,guillotkassel09}.

\subsection{Twists and gauge equivalence}
\label{sect:Jsigma}

We only set our conventions here, and refer the reader to~\cite{montgomery04,guillotkassel09,radford11,kassel12,egno15} etc. for a more detailed account.
\par

A Drinfeld twist, or just twist, for a (complete topological) Hopf algebra $H$ is a unit $J\in H\ot H$ satisfying
\[
(\Delta\ot 1)(J)(J\ot 1)=(1\ot\Delta)(J)(1\ot J)\ \ \mathrm{and}\ \ (\epsilon\ot 1)(J)=(1\ot\epsilon)(J)=1.
\]
We replace $\ot$ with $\hat{\ot}$ in the topological setting.  We let $H^J$ denote the corresponding twisted Hopf algebra, which is just $H$ as an algebra along with the new comultiplication $\Delta^J(h)=J^{-1}\Delta(h)J$, for each $h\in H$.  We denote the gauge action of $H^\times$ on the set of twists for $H$ by a dot,
\[
v\cdot J=\Delta(v)J(v^{-1}\ot v^{-1}).
\]
We say two twists $J$ and $J'$ are {\it gauge equivalent} if there is a unit $v$ so that $J'=v\cdot J$, and let $\mathrm{Tw}(H)$ denote the set of gauge equivalence classes of twists for $H$.

A {\it cocycle} twist for $H$ is defined dually, i.e. as a convolution invertible function $\sigma:H\ot H\to \C$ which satisfies
\[
\sigma(h_1h'_1,h'')\sigma(h_2,h'_2)=\sigma(h,h'_1h''_1)\sigma(h'_2,h''_2)
\]
and $\sigma(1,h)=\sigma(h,1)=1$ for each $h,h',h''\in H$.  From $\sigma$ we form the new Hopf algebra $H_\sigma$, which is the coalgebra $H$ along with the new multiplication
\[
h\cdot_\sigma h'=\sigma^{-1}(h_1,h'_1)\sigma(h_3,h'_3)h_2h'_2.
\]
There is a dual notion of gauge equivalence as well, and we let $\mathrm{CoTw}(H)$ denote the set of gauge equivalence classes of cocycle twists for $H$.
\par

Any Drinfeld twist $J$ for $H$ defines a cocycle twist for the dual $H^\ast$, and there is an equality of Hopf algebras $(H^J)^\ast=(H^\ast)_J$.

\subsection{twists for finite abelian groups}

Given a finite abelian group $\Lambda$ we let $\Lambda^\vee$ denote the character group.  For any $a\in \Lambda^\vee$ we let
\[
P_a=\frac{1}{|\Lambda|}\sum_{\lambda\in \Lambda}a(\lambda^{-1})\lambda
\]
denote the corresponding idempotent in $\C[\Lambda]$.  We will be considering bilinear forms on the character group $\Lambda^\vee$.  By this we mean set maps $B:\Lambda^\vee\times \Lambda^\vee\to \mathbb{C}^\times$ which satisfy $B(ab,c)=B(a,c)B(b,c)$ and $B(a,bc)=B(a,b)B(a,c)$, for each $a,b,c\in \Lambda^\vee$.
\par

We say a form $B$ on $\Lambda^\vee$ is {\it alternating} if $B(a,b)=B^{-1}(b,a)$ for each $a,b\in \Lambda^\vee$, and we let $\mathrm{Alt}(\Lambda^\vee)$ denote the set of alternating forms on $\Lambda^\vee$.
\par

Under the natural identification $\C[\Lambda^\vee]=\C[\Lambda]^\ast$ any form $B$ on $\Lambda^\vee$ is identified with an element $B\in \C[\Lambda]\ot\C[\Lambda]$.  This element can be written explicitly as $B=\sum_{a,b\in \Lambda^\vee}B(a,b)P_a\ot P_b$.  Bilinearity implies that any form $B$ on the character group $\Lambda^\vee$ defines a twist for $\C[\Lambda]$.  So we get a map
\[
\left\{\mathrm{Bilinear\ forms\ on\ }\Lambda^\vee\right\}\to \mathrm{Tw}(\C[\Lambda]),
\]
which restricts to a map $\mathrm{Alt}(\Lambda^\vee)\to\mathrm{Tw}(\C[\Lambda])$.

\begin{proposition}[{\cite[Cor. 5.8]{guillotkassel09}}]\label{lem:TwG}
When $\Lambda$ is a finite abelian group of odd order, the map $\mathrm{Alt}(\Lambda^\vee)\to\mathrm{Tw}(\C[\Lambda])$ is an isomorphism.
\end{proposition}

The interested reader should also consult~\cite[Prop. 4.2, Prop. 5.1, Thm. 5.5]{guillotkassel09}.

\section{Completion and duality for big quantum groups}
\label{sect:Uq}

In this section we present some relations between Lusztig's quantum Borel algebra, the small quantum Borel algebra, and the De Concini-Kac quantum Borel algebra.  In particular, we present a duality between (a completion of) Lusztig's divided powers algebra and (a modified form of) the De Concini Kac algebra.  These relations will be of significance in later sections, where we would like to make some cohomological arguments involving the duals of $U_q(\b)$ and $u_q(\b)$.  We expect that the material will be well-known to experts

\subsection{Duals and completion for $U_q(\b)$}
Take $U_q=U_q(\b)$.
\par

We consider the completion $\hat{U}_q$ of $U_q$ with respect to the $\mathbb{Z}$-grading:
\[
\hat{U}_q=\varprojlim_N U_q/(\text{elem's of degree }\geq N).
\]
Since the comultiplication on $U_q$ respects the grading, the new object $\hat{U}_q$ is a topological Hopf algebra, with comultiplication landing in the completed tensor product
\[
\Delta:\hat{U}_q\to \hat{U}_q\ \!\hat{\ot}\ \!\hat{U}_q.
\]

\begin{definition}
We let $U_q(\b)^\#$ denote the ($\mathbb{Z}$-)graded dual of $U_q(\b)$.
\end{definition}

The Hopf algebra $U_q^\#$ is graded by the root lattice $Q$, where we take $((U_q)_{-\mu})^\ast=(U_q^\#)_\mu$ for each $\mu\in Q$.  We can alternately arrive at $U_q(\b)^\#$ as the continuous dual of $\hat{U}_q$.  Whence the completed algebra $\hat{U}_q$ is naturally identified with the (non-graded, non-continuous) dual $\hat{U}_q=(U_q^\#)^\ast$.
\par

The Hopf quotient $U_q\to \C[G]$ dualizes to a Hopf embedding $\C[G^\vee]\to U_q^\#$.  Furthermore, since the coradical filtration on $U_q$ is exhaustive, with $F_0U_q=\C[G]$ and 
\[
F_1 U_q=\C[G]\oplus \left(\C[G]\cdot \{E_\alpha:\alpha\in\Gamma\}\right),
\]
we have that $U_q^\#$ is generated in degree $-1$ by the functions $(\C[G]\cdot\{ E_\alpha:\alpha\in \Gamma\})^\ast$.  We let $X_\alpha\in U_q^\#$ denote the degree $-1$ function
\[
X_\alpha:\left(\C[G]\cdot\{E_\alpha\}_\alpha\right)\to \C,\ \ K_\gamma E_\beta\mapsto \delta_{\alpha,\beta}.
\]

\begin{lemma}\label{lem:PD1}
The algebra $U_q^\#$ is generated by the grouplikes $G^\vee$ and the functions $X_\alpha$.  Each $X_\alpha$ is skew primitive, with $\Delta(X_\alpha)= X_\alpha\ot \alpha^{-1}+1\ot X_\alpha$.
\end{lemma}

We have abused notation here to identify $\alpha\in \Gamma$ with its corresponding character on $G$, via the Killing map.  Recall that for $\gamma\in G^\vee$ we let $P_\gamma=\frac{1}{|G|}\sum_{g\in G} \gamma(g^{-1})g$ denote the corresponding idempotent in $\C[G]$.

\begin{proof}
We need to show that the elements $\C[G^\vee]\{X_\alpha\}_\alpha$ produce all linear functions from $\C[G]\{E_\alpha\}_\alpha$ to $\C$.  For any character $\omega\in G^\vee$ we have
\[
\omega X_\alpha(P_\gamma E_\beta)=\sum_{ab=\gamma}\omega(K_\beta)\omega(P_a)X_\alpha(P_b E_\beta)= \omega(K_\beta)X_\alpha(P_{\omega^{-1}\gamma}E_\beta).
\]
Note that $X_\alpha(P_\tau E_\beta)$ vanishes when $\beta\neq\alpha$, for every $\tau\in G^\vee$, and that
\[
X_\alpha(P_\tau E_\alpha)=\frac{1}{|G|}\sum_{g\in G} \tau(g)=\left\{
\begin{array}{ll}
1 &\text{when }\tau\text{ is the identity in }G^\vee\\
0 &\text{else}.
\end{array}\right.
\]
Hence
\[
\omega(K_\alpha)^{-1}\omega X_\alpha(P_\gamma E_\beta)=\delta_{\omega,\gamma}\delta_{\alpha,\beta}.
\]
Since $\{P_\gamma E_\beta:\gamma\in G^\vee, \beta\in \Gamma\}$ provides a basis for the space $\C[G]\{E_\beta\}_\beta$, we see that $\C[G^\vee]\{X_\alpha\}_\alpha$ produces all degree $1$ functions.
\par

We need also to compute the comultiplication $\Delta(X_\alpha)$.  We have
\[
\Delta(X_\alpha)(K_\mu\ot E_\beta)=X_\alpha(K_\mu E_\beta)=\delta_{\alpha,\beta}
\]
and
\[
\Delta(X_\alpha)(E_\beta\ot K_\mu)=q^{-(\mu,\beta)}X_\alpha(K_\mu E_\beta)=q^{-(\mu,\alpha)}\delta_{\alpha,\beta}.
\]
Since $\Delta(X_\alpha)$ lies in bidegrees $(0,-1)$ and $(-1,0)$, the above computation show that $\Delta(X_\alpha)= X_\alpha\ot \alpha^{-1}+1\ot X_\alpha$.
\end{proof}

\begin{lemma}\label{lem:PD2}
We have $\omega X_\alpha \omega^{-1}=\omega(K_\alpha)X_\alpha$, for each $\omega\in G^\vee$ and $\alpha\in \Gamma$.
\end{lemma}

\begin{proof}
The second iteration of the comultiplication gives
\[
\omega X_\alpha \omega^{-1}(K_\gamma E_\beta)=\omega(K_\gamma K_\beta)\omega(K_\gamma)^{-1}\delta_{\alpha\beta}=\omega(K_\alpha)\delta_{\alpha,\beta}=\omega(K_\alpha)X_\alpha(K_\gamma E_\beta),
\]
for all $\gamma,\beta$.
\end{proof}

\subsection{Identification with the modified De Concini-Kac algebra}
\label{sect:modDK}

Recall the De Concini-Kac algebra
\[
U^{DK}_{q^{-1}}(\b_-)=\frac{\C\langle K_\alpha,F_\alpha:\alpha\in \Gamma\rangle}{\left(\begin{array}{c}
[K_{\alpha},K_{\beta}]=0,\ K_{\alpha}F_\beta=q^{(\alpha,\beta)}F_\beta K_\alpha,\ K_\alpha^l=1,\\
\sum_{i=0}^{1-\langle\alpha,\beta\rangle}(-1)^i\qbinom{1-\langle\alpha,\beta\rangle}{i}_{q^{-1}_\alpha}F_\alpha^{1-\langle\alpha,\beta\rangle-i}F_\beta F_\alpha^i=0
\end{array}\right)}
\]
~\cite{deconcinikac89}.  The $K_\alpha$ here are grouplike and
\[
\Delta(F_\alpha)=F_\alpha\ot K_\alpha^{-1}+ 1\ot F_\alpha.
\]
The algebra $U^{DK}_{q^{-1}}$ is graded by the root lattice, with $K_\alpha$ of degree $0$ and $F_\alpha$ of degree $-\alpha$.
\par

There is a factoring $U^{DK}_{q^{-1}}=U^-_{q^{-1}}\rtimes \C[G]$, where $U^-_{q^{-1}}$ is the subalgebra generated by the $F_\alpha$.  We define the {\it modified De Concini-Kac algebra} as the smash product
\[
\U^{DK}=\U^{DK}_{q^{-1}}=U^-_{q^{-1}}\rtimes \C[G^\vee],
\]
where $\C[G^\vee]$ acts on the negative subalgebra in the natural manner $\omega\cdot F_\alpha=\omega(K_\alpha)F_\alpha$.  We give $\U^{DK}$ the obvious Hopf structure, where $G^\vee$ consists of grouplikes and the $\Delta(F_\alpha)=F_\alpha\ot \alpha^{-1}+1\ot F_\alpha$.
\par

Most of the properties of $U^{DK}_{q^{-1}}$ will be inherited by the modified algebra.  For example, $\U^{DK}$ is $Q$-graded in the obvious way, and Lusztig's basis for $U^-_{q^{-1}}$ produces a basis for $\U^{DK}$.

\begin{lemma}\label{lem:DKD}
There is an isomorphism
\[
\varphi:\U^{DK}\to U_q(\b)^\#,\ \ \begin{array}{l}
F_\alpha\mapsto X_\alpha\\
\omega\mapsto \omega,
\end{array}
\]
of $Q$-graded Hopf algbras.
\end{lemma}

\begin{proof}[Sketch proof]
We claim first that $U_q(\b)^\#$ satisfies all of the relations of the modified De Concini-Kac algebra.  The only relations left to check are the Serre relations, by Lemma~\ref{lem:PD2}.  In the simply laced case one can readily check the relations by hand.  In the non-simply laced case the determinant of the Cartan matrix will be $1$ or $2$, and hence invertible in $\mathbb{Z}/l\mathbb{Z}$.  The small quantum group $u_q(\g)$ will then be factorizable with $R$-matrix
\[
R=\prod_{\mu\in\Phi^+}\left(\sum_{n=0}^{l-1}q^{-n(n+1)/2}\frac{(1-q^2)^n}{[n]_q!} E^n_{\mu}\ot F^n_{\mu}\right)\Omega,
\]
where $\Omega\in \C[G]\ot \C[G]$ is the inverse of the Killing form (see e.g.~\cite{turaev}).  The $R$-matrix provides an isomorphism $u_q(\b)^\ast\to u_q(\b_-)^{cop}\cong u_{q^{-1}}(\b_-)$ which sends each generator $X_\alpha$ to $q^{2}(q^{-1}-q)F_\alpha$~\cite{radford93}.  Hence the Serre relations hold for $U_q(\b)^\#$ modulo the kernel of the projection $U_q(\b)^\#\to u_q(\b)^\ast$.
\par

Since the inclusion $u_q(\b)\to U_q(\b)$ is an isomorphism in all $\mathbb{Z}$-degrees $<l$, it follows that the graded dual is an isomorphism in all degrees $>-l$.  By our assumptions on $l$, the Serre relations are in degrees $>-l$, and therefore must hold for $U_q(\b)^\#$.  It follows that the algebra map $\varphi$ can be defined in the prescribed manner.
\par

By Lemma~\ref{lem:PD1} $\varphi$ is surjective.  Also, Lusztig's bases~\cite{deconcinikac89,lusztig90,lusztig90II} of both $\U^{DK}$ and $U_q$ tell us that the two algebras have the same dimension in each $Q$-degree.  So $\varphi$ must be an isomorphism.  The fact that $\varphi$ is a Hopf map follows from the fact that comultiplications agree on the generators
\[
(\varphi\ot\varphi)\Delta(F_\alpha)=X_\alpha\ot \alpha^{-1}+1\ot X_\alpha=\Delta(X_\alpha).
\]
\end{proof}

\section{Interactions between Drinfeld twists and cohomology}
\label{sect:drinfH}

Fix $\hat{U}_q=\hat{U}_q(\b)$, $u_q=u_q(\b)$, and $\U^{DK}=\U^{DK}_{q^{-1}}$ the modified De Concini-Kac algebra of Section~\ref{sect:modDK}.  We identify $\U^{DK}$ with the continuous dual of $\hat{U}_q$ via the isomorphism of Lemma~\ref{lem:DKD}.
\par

In this section we cover an association between twists for $\hat{U}_q$ (resp. $u_q$) and cohomology classes for $\U^{DK}$ (resp. $u_q^\ast$).  Namely we show that minimal $Q$-homogenous terms $J_\gamma$ in a twists $J$ produce cohomology classes $[J_\gamma]\in H^2(\U^{DK}_B)$ (resp. $[J_\gamma]\in H^2((u_q^\ast)_B)$), where $B$ is a form on $G^\vee$ associated to $J$.  The class is show to vanish exactly when the term $J_\gamma$ in $J$ can be eliminated via gauge transformation.

\subsection{Basics for twists of $u_q$ and $\hat{U}_q$}

\begin{lemma}\label{lem:469}
Any twist for $u_q$, or $\hat{U}_q$, is gauge equivalent to one of the form $J=B+\sum_{\mu\in Q^+}J_\mu$, where $B\in \C[G]\ot\C[G]$ is an alternating bilinear form on $G^\vee$ and each $J_\mu$ is homogenous of $Q$-degree $\mu$.
\end{lemma}

\begin{proof}
We may expand any twist as $J'=J'_0+\sum_{\mu\in Q^+}J'_\mu$.  The degree $0$ term $J'_0$ will be a twist for $\C[G]$, for grading reasons, and by~\cite[Cor. 5.8]{guillotkassel09} there is a unit $v\in \C[G]$ so that $v\cdot J'_0=B$ is an alternating bilinear form on $G^\vee$.  Hence $J=v\cdot J'$ will have the prescribed form.
\end{proof}

\begin{remark}
We will generally denote the degree $0$ term in a twist $J$ by $B$, although $B$ needn't always restrict to a bilinear form on $G^\vee$.  As we are interested in twists up to gauge equivalence, one is free to always assume $B$ is a form on $G^\vee$ in any case.
\end{remark}

Since any twist $B$ for $\C[G]$ is also a twist for the ambient algebra $\hat{U}_q$, we may consider the corresponding twisted Hopf algebra $\hat{U}_q^B$.  Recall that $\hat{U}_q=\hat{U}_q^B$ as algebras, and hence the corresponding groups of units are also identified, $(\hat{U}_q)^\times=(\hat{U}_q^B)^\times$.

\begin{lemma}\label{lem:Uset}
Given a twist $B$ for $\C[G]$, left multiplication by $B^{\pm 1}$ provides an isomorphism of $(\hat{U}_q)^\times$-sets
\[
\xymatrix{
\{\mathrm{Drinfeld\ twists\ of\ }\hat{U}_q\}\ar@/^.2pc/[r]^{B^{-1}} & \{\mathrm{Drinfeld\ twists\ of\ }\hat{U}_q^B\},\ar@/^.2pc/[l]^B
}
\]
under the gauge actions.  Furthermore, if a subgroup $\mathbb{V}\subset (\hat{U}_q)^\times$ preserves the subset of twists for $u_q$, under the gauge action, then $\mathbb{V}$ also preserves twists for $u_q^B$ and we have an isomorphism of $\mathbb{V}$-sets
\[
\xymatrix{
\{\mathrm{Drinfeld\ twists\ of\ }u_q\}\ar@/^.2pc/[r]^{B^{-1}} & \{\mathrm{Drinfeld\ twists\ of\ }u_q^B\}.\ar@/^.2pc/[l]^B
}
\]
\end{lemma}

\begin{proof}
We first prove that there is a bijection of sets for $\hat{U}_q$, without considering the gauge action.  Suppose $J$ is a twist for $\hat{U}_q$.  One simply calculates:
\[
\begin{array}{rl}
(\Delta^B\ot 1)(B^{-1}J)(B^{-1}J\ot 1)&
=(B^{-1}\ot 1)(\Delta\ot 1)(B^{-1})(\Delta\ot 1)(J)(J\ot 1)\\
&=(1\ot B^{-1})(1\ot\Delta)(B^{-1})(1\ot \Delta)(J)(1\ot J)\\
&=(1\ot \Delta^B)(B^{-1}J)(1\ot B^{-1}J).
\end{array}
\]
So $B^{-1}J$ is a twist for $\hat{U}_q^B$.  A completely similarly calculation shows that $BF$ is a twist for $\hat{U}_q$ whenever $F$ is a twist for $\hat{U}_q^B$.  Compatibility with the action of $(\hat{U}_q)^\times$ follows from the equalities
\[
B^{-1}\Delta(v)J(v^{-1}\ot v^{-1})=\Delta^B(v)(B^{-1}J)(v^{-1}\ot v^{-1})
\]
and $B\Delta^B(v)F(v^{-1}\ot v^{-1})=\Delta(v)(BF)(v^{-1}\ot v^{-1})$.  The statement for $u_q$ follows by the above computations and the fact that $B$ is a unit in $u_q\ot u_q$.
\end{proof}

\subsection{Elimination of low degree terms and cohomology}
\label{sect:506}

Given a twist $J'=B+\sum_{\mu}J'_\mu$ for $\hat{U}_q$, we can multiply by $B^{-1}$ to get a twist $J=1+\sum_{\mu\in Q^+} J_\mu$ for $\hat{U}_q^B$.  Hence, considering arbitrary twists for $\hat{U}_q$ is equivalent to considering twists of the form $J=1+\sum_{\mu\in Q^+}J_\mu$ for the algebras $\hat{U}_q^B$, where $B\in \C[G]\otimes \C[G]$.  Fix such a twist $J\in \hat{U}_q^B\ot \hat{U}^B_q$ and suppose $\gamma$ is of minimal $Q^+$-degree such that $J_\gamma\neq 0$.
\par

The equation
\[
(\Delta^B\ot 1)(J)(J\ot 1)=(1\ot \Delta^B)(J)(1\ot J)
\]
produces in degree $\gamma$ the equation
\begin{equation}\label{eq:cocycle}
(1\ot J_\gamma)-(\Delta^B\ot 1)(J_\gamma)+(1\ot \Delta^B)(J_\gamma)-(J_\gamma\ot 1)=0.
\end{equation}
Furthermore, a unit $v=1+\sum_{\mu\in Q^+}v_\mu$ in $\hat{U}_q^B$ is such that the gauged twist $v\cdot J$ has vanishing term in degree $\gamma$ if and only if $v_\gamma$ solves the equation
\begin{equation}\label{eq:cobdry}
(1\ot v_\gamma)-\Delta^B(v_\gamma)+(v_\gamma\ot 1)=J_\gamma.
\end{equation}
We call equation~\eqref{eq:cocycle} the cocycle equation and equation~\eqref{eq:cobdry} the coboundary equation.  We call a solution $v_\gamma$ to~\eqref{eq:cobdry} a {\it bounding element} for $J_\gamma$.

\begin{lemma}\label{lem:bleh}
Suppose $J=1+\sum_{\mu\in Q^+}J_\mu$ is a twist for $\hat{U}_q^B$ and $\gamma\in Q^+$ is a minimal degree such that $J_\gamma\neq 0$.  Consider any unit $v=1+\sum_{\mu\in Q^+} v_\mu$ with $v_\gamma$ solving~\eqref{eq:cobdry} and $v_\nu=0$ whenever $J_{\nu'}=0$ for all $\nu'\leq \nu$.  Then the gauged twist $(v\cdot J)$ has the following properties: 
\begin{enumerate}
\item[(i)] $(v\cdot J)_\gamma=0$.
\item[(ii)] $(v\cdot J)_\nu=0$ whenever $J_{\nu'}=0$ for all $\nu'\leq \nu$.
\end{enumerate}
In particular, the above conclusions hold when $v_\gamma$ solves~\eqref{eq:cobdry} and we take $v=\exp(v_\gamma)$.
\end{lemma}

\begin{proof}
This follows from a direct analysis of the equation
\[
\begin{array}{rl}
v\cdot J&=J_\gamma-(1\ot v_\gamma-\Delta^B(v_\gamma)+v_\gamma\ot 1)\\
&\ \ +\sum_{\mu\neq \gamma}{\left(J_\mu-(1\ot v_\mu-\Delta^B(v_\mu)+v_\mu\ot 1)\right)}\\
&\ \ +\sum_I\Delta^B(v_a)J_b(v_c^{-1}\ot v_d^{-1})\\\\
&=\sum_{\mu\neq \gamma}{\left(J_\mu-(1\ot v_\mu-\Delta^B(v_\mu)+v_\mu\ot 1)\right)}\\
&\ \ +\sum_I\Delta^B(v_a)J_b(v_c^{-1}\ot v_d^{-1}),
\end{array}
\]
where the final sum is indexed over the set $I$ of tuples $(a,b,c,d)\in Q^4$ with at least two nonzero entries.
\end{proof}

The lemma implies that if we can solve~\eqref{eq:cobdry}, for a given minimal term $J_\gamma$ in a given twist $J$, then we can get rid of the term $J_\gamma$ via gauge transformation.  Furthermore, we can do so in a manner which does not introduce new terms in low degree.
\par

Consider again a twist $J=1+\sum_{\mu\in Q^+}J_\mu$ for $\hat{U}^B_q$, and take $J_\gamma$ a nonvanishing term of minimal $Q$-degree.  If we consider $J_\gamma$ as a function on the cocycle twisted De Concini-Kac algebra $\U^{DK}_B$, and apply~\eqref{eq:cocycle} to elements $X\ot Y\ot Z$ in $\U^{DK}_B$, we arrive at the equation
\[
\epsilon(X)J_\gamma(Y,Z)-J_\gamma(X\cdot_B Y,Z)+J_\gamma(X,Y\cdot_B Z)-J_\gamma(X,Y)\epsilon(Z)=0.
\]
Similarly, the coboundary equation becomes
\begin{equation}\label{eq:514}
\epsilon(X)v_\gamma(Y)-v_\gamma(X\cdot_B Y)+v_\gamma(X)\epsilon(Y)=J_\gamma(X,Y).
\end{equation}
So, if we let
\[
C^\bt(\U^{DK}_B)=0\to k\to \Hom_k(\U^{DK}_B,k)\to \Hom_k(\U^{DK}_B\ot \U^{DK}_B,k)\to\cdots
\]
denote the standard complex computing Hopf cohomology $H^\bt(\U^{DK}_B)=\Ext_{\U^{DK}_B}(k,k)$, then $J_\gamma$ is a degree $2$-cocycle in $C^\bt(\U^{DK}_B)$.  By equation~\eqref{eq:514}, the function $J_\gamma$ is bounded by some $v_\gamma$ if and only if the class $[J_\gamma]$ vanishes in cohomology.
\par

A completely similar analysis holds for the small quantum Borel $u_q^B$.  We record these facts in a lemma.

\begin{lemma}
Let $B$ be a twist for $\C[G]$ and $J=1+\sum_{\mu\in Q^+}J_\mu$ be a twist for $\hat{U}^B_q$ (resp. $u_q^B$).  If $\gamma$ is a minimal $Q$-degree so that $(J-1)_\gamma$ is nonvanishing, then the corresponding function
\[
J_\gamma\in C^2(\U^{DK}_B)\ \ \left(\text{resp. }J_\gamma\in C^2((u_q^\ast)_B)\right)
\]
is a cocycle.  There is a homogenous element $v_\gamma$ in $U^B_q$ (resp. $u_q^B$) bounding $J_\gamma$, in the sense described above, if and only if the cohomology class $[J_\gamma]\in H^2(\U^{DK}_B)$ (resp. $[J_\gamma]\in H^2((u_q^\ast)_B)$) vanishes.
\end{lemma}

\section{Cohomologies of $\U^{DK}_B$ and $(u^\ast_q)_B$}
\label{sect:cohom}

We provide here a description of the cohomologies $H^i(\U^{DK}_B)=\Ext_{\U^{DK}_B}^i(k,k)$ and $H^i((u^\ast_q)_B)=\Ext_{(u^\ast_q)_B}(k,k)$ when $B$ is a bilinear form on $G^\vee$, and $i=1,2$.  The proofs of the most important results, Propositions~\ref{prop:HUB} and~\ref{prop:HuB}, are delayed until Section~\ref{sect:UDK}, as they are something of a deviation from our primary objective of classifying twists.

\subsection{Cohomology for De Concini-Kac algebras}

Our main computation for the De Concini-Kac algebra is

\begin{proposition}\label{prop:HUB}
For any bilinear form $B$ on $G^\vee$ we have $H^1(\U^{DK}_B)=0$ and $H^2(\U^{DK}_B)=0$.
\end{proposition}

As remarked above, the proof of Proposition~\ref{prop:HUB} will be given in Section~\ref{sect:UDK}.  Of course, we are mostly interested in the vanishing of the second cohomology $H^2(\U^{DK}_B)$.  However, vanishing of $H^1(\U^{DK}_B)$ tells us something which will be quite useful in our classification of twists for the small algebra $u_q$.

\begin{corollary}
Consider any form $B$ on $G^\vee$ and twist $J=B+\sum_{\mu\in Q^+}J_\mu$ of the completion $\hat{U}_q$.  Then for any nonvanishing term $J_\gamma$ of minimal $Q$-degree, there exists a unique bounding element $v_\gamma\in \hat{U}_q$.
\end{corollary}

\begin{proof}
Vanishing of $H^2(\U^{DK}_B)$ implies the existence of such a $v_\gamma$.  Vanishing of $H^1(\U^{DK}_B)$ tells us that the map
\[
v_\gamma\mapsto
(1\ot v_\gamma)-\Delta^B(v_\gamma)+(v_\gamma\ot 1)
\]
is injective.  Hence the bounding element $v_\gamma$ is unique.
\end{proof}

As a consequence of Proposition~\ref{prop:HUB} we can deduce a classification of cocycle twists for $\U^{DK}$.  Note that we have a canonical map $\mathrm{CoTw}(\C[G^\vee])\to \mathrm{CoTw}(\U^{DK})$ given by restricting along the Hopf projection
\[
\U^{DK}\ot \U^{DK}\to \C[G^\vee]\ot \C[G^\vee].
\]
Recall also the identification $\mathrm{CoTw}(\C[G^\vee])=\mathrm{Alt}(G^\vee)$ of Section~\ref{sect:twists}.  We expect that the following result is known, although we do not know a reference.

\begin{proposition}\label{prop:triv1}
For the modified De Concini-Kac algebra $\U^{DK}$ the restriction map $\mathrm{Alt}(G^\vee)\to \mathrm{CoTw}(\U^{DK})$ is a bijection.
\end{proposition}

\begin{proof}
The map $\mathrm{CoTw}(\U^{DK})\to \mathrm{Tw}(\hat{U}_q)$ given by the identification $\hat{U}_q=(\U^{DK})^\ast$ is a bijection.  We show that the map $\mathrm{Alt}(G^\vee)\to \mathrm{Tw}(\hat{U}_q)$ is an isomorphism.
\par

Consider a twist $J'$ for $\hat{U}_q$.  After gauging by an element in $\C[G]$ we may assume the degree zero term for $J'$ is a form $B$ on $G^\vee$.  Consider the associated twist $J=B^{-1}J'=1+\sum_{\mu\in Q^+}J_\mu$ for $U^B_q$.
\par

Take an arbitrary enumeration $\{\mu_1,\mu_2,\dots\}$ of $Q^+$ which is compatible with the total ordering on $\mathbb{Z}$.  Specifically, we demand $|\mu_i|\leq |\mu_j|$ whenever $i\leq j$.  Define inductively $J(0)=J$ and $J(i)=\exp(v_{\mu_i})\cdot J(i-1)$, where $v_{\mu_i}$ is the unique bounding element for $J(i-1)_{\mu_i}$.  Then, by Lemma~\ref{lem:bleh}, the (convergent) ordered product $v=\prod_{i=1}^\infty \exp(v_{\mu_i})$ will be such that $v\cdot_B J=1$.  Hence, by Lemma~\ref{lem:Uset}, we have $v\cdot J'=B$.  So we see that the map $\mathrm{Alt}(G^\vee)\to \mathrm{Tw}(\hat{U}_q)$ is surjective.  Injectivity can be deduced from Lemma~\ref{lem:TwG}.
\end{proof}

\begin{remark}
The proof of Proposition~\ref{prop:HUB} works perfectly well for the non-modified De Concini-Kac Borel algebra.  Hence we will have vanishing of cohomologies $H^1(U^{DK}_T)=H^2(U^{DK}_T)=0$, when $T$ is a form on $G$.  We therefore get a bijection
\[
\mathrm{Alt}(G)\to \mathrm{CoTw}(U^{DK}_T),
\]
as above.  The role of $\hat{U}_q$ in this case is be played by a modified algebra $\hat{\U}_q$.
\end{remark}

\subsection{Cohomology for small quantum Borel algebras}
\label{sect:HuB}

Let us now consider the small algebra $u_q=u_q(\b)$.  Note that the cohomologies $H^i((u^\ast_q)_B)$ will naturally be graded by the root lattice.  One can see this by observing that the $Q$-grading on $u_q^B$ induces a $Q$-grading on the complex
\[
C^\bt\left((u^\ast_q)_B\right)=0\to k\to \Hom_k((u^\ast_q)_B,k)\to \Hom_k((u^\ast_q)_B\ot (u^\ast_q)_B,k)\to\cdots
\]
\[
=0\to k\to u_q^B\to u_q^B\ot u_q^B\to \cdots.
\]
The next result follows by~\cite{ginzburgkumar93} in the specific case $B=1$.  As with Proposition~\ref{prop:HUB}, we provide the proof in Section~\ref{sect:UDK}.

\begin{proposition}[{cf.~\cite[Thm. 2.5]{ginzburgkumar93}}]\label{prop:HuB}
Let $B$ be any bilinear form on $G^\vee$.  Let $\n^{(l)}=\oplus_{\mu\in \Phi^+}\C e_\mu$ denote the positive nilpotent subalgerba in $\g$ equipped with the $Q$-grading $\deg_Q(e_\mu)=l\mu$.  Then we have $H^1\left((u_q^\ast)_B\right)=0$ and $H^2((u_q^\ast)_B)=\n^{(l)}$ as $Q$-graded vector spaces.
\end{proposition}

The identification $H^2((u^\ast_q)_B)=\n^{(l)}$ suggests that there should be non-trivial twists $J$ of the form $J=B+J_{l\mu}+(\mathrm{higher\ deg\ terms})$ for each $\mu\in \Phi^+$.  Indeed, we can construct such a twist, for $\mu=\alpha$ for example, via gauge transformation in $\hat{U}_q$, $J=\exp(E_\alpha^{(l)})\cdot B$.  (See Lemma~\ref{lem:vJ} below.)  In Section~\ref{sect:classTw} we will show that in fact {\it all} homogeneous classes in $H^2((u^\ast_q)_B)$ arise as, and can be eliminated by, gauge transformations by such exponents $\exp(\lambda E^{(l)}_\alpha)$ in $\hat{U}_q$.

\section{Classifying twists for $u_q(\b)$}
\label{sect:classTw}

The primary result of this section is Theorem~\ref{thm:twists} below.  Theorem~\ref{thm:twists}, along with some information from~\cite{uqJ}, will imply Theorem A from the introduction.  Recall our notation $u_q=u_q(\b)$, $U_q=U_q(\b)$.

\subsection{The gauge group $\mathbb{V}$}
\label{sect:V}

Let $\mathbb{V}_{small}$ and $\mathbb{V}_{dp}$ be the following two subgroups of $(\hat{U}_q)^\times$:
\[
\mathbb{V}_{small}:=\langle 1+v_+:v_+\in (u_q)_{\geq 1}\rangle,
\]
\[
\mathbb{V}_{dp}:=\langle \exp(\lambda E_\alpha^{(l)}):\alpha\in \Gamma,\ \lambda\in\C\rangle.
\]
Note that $V_{small}\subset (u_q)^\times$.

\begin{definition}
Let $\mathbb{V}$ denote the subgroup of $(\hat{U}_q)^{\times}$ generated by $\mathbb{V}_{small}$ and $\mathbb{V}_{dp}$.
\end{definition}

We can write any $v\in \mathbb{V}$ as a sum $v=1+\sum_{\mu\in Q^+}v_\mu$ with each $v_\mu$ homogenous of $Q$-degree $\mu$, since we can do so for each of the generators of $\mathbb{V}$.  It will be helpful to have the following lemma from~\cite{lusztig89}.

\begin{lemma}[{\cite[Lem. 4.5]{lusztig89}}]\label{lem:exp}
For each $\alpha\in \Gamma$ and $\lambda\in \mathbb{C}$ the derivation $\mathrm{ad}_{\lambda E^{(l)}_\alpha}:\hat{U}_q\to \hat{U}_q$ restricts to a nilpotent derivation on $u_q$, and the exponential $\exp(\mathrm{ad}_{\lambda E^{(l)}_{\alpha}})$ restricts to a well-defined algebra automorphism of $u_q$.
\end{lemma}

The first claim follows directly from~\cite{lusztig89}.  The second claim follows from the general fact that exponentiating a nilpotent derivation, of any given algebra, produces an algebra automorphism.

\begin{lemma}\label{lem:vJ}
Fix a twist $B$ for $\C[G]$.  Given any twist $J$ for $u_q^B$, and $v\in \mathbb{V}$, the gauging $(v\cdot J)$ is another twist for $u_q^B$.
\end{lemma}

\begin{proof}
We prove the result in the case $B=1$. The case when $B$ is arbitrary will follow by Lemma~\ref{lem:Uset}.
\par

Suppose $B=1$.  Obviously such twists are preserved by the action of $\mathbb{V}_{small}$, so it suffices to prove that $(v\cdot J)$ is a twist for $u_q$ whenever $v$ is a generator $v=\exp(\lambda E_\alpha^{(l)})$ of $\mathbb{V}_{dp}$.  We have $\Delta(\lambda E_\alpha^{(l)})=\lambda E_\alpha^{(l)}\ot 1+1\ot \lambda E^{(l)}_\alpha+T$, where $T$ is a square-zero element in $u_q$ which centralizes $E_\alpha^{(l)}\ot 1$ and $1\ot E_\alpha^{(l)}$.  One can write explicitly $T=\lambda\sum_{i=1}^{l-1} q^{-i(l-i)}K_\alpha^iE_\alpha^{(l-i)}\ot E_\alpha^{(i)}$.  Then we have 
\[
\Delta(v)=\exp(\lambda\Delta(E_\alpha^{(l)}))=(1+ T)(v\ot v).
\]
Hence 
\[
\Delta(v)J(v^{-1}\ot v^{-1})=(1+ T)(v\ot v)J(v^{-1}\ot v^{-1}).
\]
Now we use the equality
\[
(v\ot v)J(v^{-1}\ot v^{-1})=\left((\exp \mathrm{ad}_{\lambda E^{(l)}_\alpha})\ot (\exp\mathrm{ad}_{\lambda E^{(l)}_\alpha})\right)(J),
\]
and Lemma~\ref{lem:exp}, to see that $v\cdot J\in u_q\ot u_q$.
\end{proof}

\begin{lemma}\label{lem:normal}
The subgroup $\mathbb{V}_{small}\subset \mathbb{V}$ is normal.
\end{lemma}

\begin{proof}
Consider any $v=1+v_+$ in $\mathbb{V}_{small}$ and any generator $\exp(\lambda E_\alpha^{(l)})$ of $\mathbb{V}_{dp}$.  Note that $\mathrm{ad}_{\lambda E^{(l)}_\alpha}$ preserves the subspace of positive degree elements in $u_q$.  Hence $\exp(\mathrm{ad}_{\lambda E^{(l)}_\alpha})(v_+)$ is of positive degrees.  It follows that the commutator
\[
\exp(\lambda E_\alpha^{(l)})v\exp(-\lambda E_\alpha^{(l)})=1+\exp(\mathrm{ad}_{\lambda E^{(l)}_\alpha})(v_+)
\]
is in $\mathbb{V}_{small}$.  Since $\mathbb{V}$ is generated by $\mathbb{V}_{small}$ and the $\exp(\lambda E_\alpha^{(l)})$, it follows that $\mathbb{V}_{small}$ is normal.
\end{proof}

\subsection{A classification of twists for $u_q(\b)$}

We wish to prove the following result.

\begin{theorem}\label{thm:twists}
Any twist $J$ for $u_q(\b)$ is gauge equivalent to a twist of the form $v\cdot B$, where $v$ is in $\mathbb{V}_{dp}$ and $B$ is an alternating form on $G^\vee$.
\end{theorem}

We break the proof down into a sequence of smaller results.  For a given form $B$ on $G^\vee$ we let $d_B$ denote the differential on either of the complexes $C^\bt(\U^{DK}_B)$ or $C^\bt((u_q^\ast)_B)$.  (Since the restriction map $C^\bt((u_q^\ast)_B)\to C^\bt(\U^{DK}_B)$ is an inclusion of cochains the notation is unambiguous.)  Recall that in each degree $i$ we have $C^i((u_q^\ast)_B)=(u_q^B)^{\ot i}$ and $C^i(\U^{DK}_B)=(\hat{U}_q^B)^{\hat{\ot}i}$.

\begin{lemma}\label{lem:dBv}
Let $B$ be a bilinear form on $G^\vee$.  If an element $v_\gamma\in \hat{U}_q\setminus u_q$ is such that $d_B(v_\gamma)\in u_q\ot u_q$, then the corresponding class $[d_B(v_\gamma)]\in H^2((u_q^\ast)_B)$ is nonzero.
\end{lemma}

\begin{proof}
We proceed by contradiction.  If $[d_B(v_\gamma)]=0$ then there is an element $v'_\gamma\in u_q$ with $d_B(v'_\gamma)=d_B(v_\gamma)$.  Since $v_\gamma\notin u_q$ the difference $v_\gamma-v'_\gamma$ provides a nonzero element in $C^1(\U^{DK}_B)$ with $d_B(v_\gamma-v'_\gamma)=0$.  But this cannot happen as the difference would provide a nonvanishing class in the first cohomology $H^1(\U^{DK}_B)$, which we know is $0$ by Proposition~\ref{prop:HUB}.
\end{proof}

Recall the exact sequence of Hopf algebras
\[
1\to u_q(\b)\to U_q(\b)\to U(\n)\to 1
\]
from Section~\ref{sect:Lus}, where $\n=\oplus_{\mu\in \Phi^+} \C e_\mu$ is the positive nilpotent subalgebra in $\b$.

\begin{lemma}\label{lem:vmu}
For any positive root $\mu\in \Phi^+$, and $\gamma=l\mu$, there is an element $v_{\gamma}\in \hat{U}_q\setminus u_q$ of $Q$-degree $\gamma$ which extends to a family of units $\mathbf{v}_{\gamma}(\lambda)$ in $\mathbb{V}_{dp}$ of the form
\[
\mathbf{v}_{\gamma}(\lambda)=1+\lambda v_{\gamma}+(\mathrm{terms\ of\ strictly\ higher\ }Q\text{-}\mathrm{degree}).
\]
\end{lemma}

\begin{proof}
We may write each generator $e_\mu$ of the root space in $\n$ as 
\[
e_\mu=[e_{\alpha_1},[\dots[e_{\alpha_{n-1}},e_{\alpha_n}]\dots]],
\]
for some sequence of simple roots $\alpha_1,\dots,\alpha_n$.  We claim that for each $\mu\in \Phi^+$ there is a family of units of the form
\[
\mathbf{v}_{l\mu}(\lambda)=1+\lambda [E^{(l)}_{\alpha_1},[\dots [E^{(l)}_{\alpha_{n-1}},E^{(l)}_{\alpha_n}]\dots]]+(\text{higher $Q$-degree terms})
\]
in $\mathbb{V}_{dp}$.  We proceed by induction on the length $n$ of the given sequence, i.e. the height of the root.  For our base case $n=1$ we take $\mathbf{v}_{l\alpha}(\lambda)=\exp(\lambda E_\alpha^{(l)})$.
\par

Suppose the result holds for height $n$ roots and that $\nu$ is height $n+1$.  Write $e_{\nu}=[e_{\alpha},e_\mu]$ for a simple root $\alpha$ and $e_\mu$ as above, and consider 
\[
v=\mathbf{v}_{l\mu}(1)=1+[E^{(l)}_{\alpha_1},[\dots [E^{(l)}_{\alpha_{n-1}},E^{(l)}_{\alpha_n}]\dots]]+(\text{higher $Q$-degree terms}).
\]
Take $v_+=v-1$.  One now calculates
\[
\begin{array}{rl}
\exp(\lambda E^{(l)}_{\alpha})v\exp(\lambda E^{(l)}_{\alpha})v^{-1}&=\exp(\lambda \mathrm{ad}_{E^{(l)}_\alpha})(v)v^{-1}\\
&=(v+(\exp(\lambda \mathrm{ad}_{E^{(l)}_\alpha})-id)(v))v^{-1}\\
&=1+(\exp(\lambda \mathrm{ad}_{E^{(l)}_\alpha})-id)(1+v_+)v^{-1}\\
&=1+(\exp(\lambda \mathrm{ad}_{E^{(l)}_\alpha})-id)(v_+)v^{-1}\\
&=1+\left(\lambda[E_\alpha^{(l)},v_{l\mu}]+\text{higher $Q$-degree terms}\right)v^{-1}.
\end{array}
\]
Since $v^{-1}=1-(\text{terms of degree }\geq l\mu)$ the final expression reduces to give
\[
\exp(\lambda E^{(l)}_{\alpha})v\exp(\lambda E^{(l)}_{\alpha})v^{-1}=1+\lambda[E_\alpha^{(l)},v_{l\mu}]+\text{higher $Q$-degree terms}.
\]
Hence
\[
\mathbf{v}_{l\mu}(\lambda):=\exp(\lambda E^{(l)}_{\alpha})v\exp(\lambda E^{(l)}_{\alpha})v^{-1}
\]
provides the desired family, with 
\[
v_{l\nu}=[E_\alpha^{(l)},v_{l\mu}]=[E^{(l)}_\alpha,[E^{(l)}_{\alpha_1},[\dots [E^{(l)}_{\alpha_{n-1}},E^{(l)}_{\alpha_n}]\dots]]].
\]
To see that the element $v_{l\nu}\in U_q\subset\hat{U}_q$ is not in $u_q$ note that its image in $U(\n)$ is $e_\nu(\neq 0)$.
\end{proof}

We use the elements $v_\gamma$ from the above lemma, in conjunction with Lemma~\ref{lem:dBv}, to construct all nonvanishing classes in $H^2((u_q^\ast)_B)$.

\begin{lemma}\label{lem:elimination}
Consider a form $B$ on $G^\vee$ and a twist $J=B+\sum_{\mu\in Q^+}J_\mu$ for $u_q$.  Let $d$ be the minimal positive integer with $J_d\neq 0$, in the $\mbb{Z}$-grading.  Then there is a unit $v\in \mathbb{V}$ so that $(v\cdot J)_{d'}=0$ for each $0<d'\leq d$.
\end{lemma}

\begin{proof}
Let us consider instead a twist $J=1+\sum_{\mu\in Q^+}J_\mu$ for $u_q^B$, with $B$ a form on $G^\vee$.  Take $d$ the minimal $\mathbb{Z}$-degree in which $J-1$ does not vanish.  We claim that there exists $v\in \mathbb{V}$ so that $(v\cdot J)-1$ vanishes in all degrees $\leq d$.  To show this it suffices to show that we can eliminate each term $J_\gamma$ with $|\gamma|=d$, i.e. that for each nonvanishing $J_\gamma$ with $|\gamma|=d$ there exists a unit $v\in \mathbb{V}$ so that $(v\cdot J)_{d''}=0$ for $0<d''<d$
and $(v\cdot J)_d=J_d-J_\gamma$.
\par

For any such $J_\gamma$ we consider the corresponding class in cohomology $[J_\gamma]\in H^2((u_q^\ast)_B)$.  If $[J_\gamma]=0$ then there is an element $v_\gamma\in u_q^B$ of degree $\gamma$ with $d_B(v_\gamma)=J_\gamma$.  Consequently, the unit $v=\exp(v_\gamma)\in \mathbb{V}_{small}$ is such that
\[
(v\cdot J)_{d''}=0\text{ for }0<d''<d
\]
and $(v\cdot J)_d=J_d-J_\gamma$, by Lemma~\ref{lem:bleh}.  If $[J_\gamma]\neq 0$ then $\gamma=l\mu$ for some positive root $\mu$, since $H^2((u_q^\ast)_B)$ is concentrated in $Q$-degrees $l\cdot \Phi^+$ by Proposition~\ref{prop:HuB}.
\par

Consider in this case the family $\mathbf{v}_{\gamma}(\lambda)=1+\lambda v_{\gamma}+\dots$ of Lemma~\ref{lem:vmu}.  We know that each
\[
\begin{array}{rl}
\mathbf{v}_{\gamma}(\lambda)\cdot (1\ot 1)&=\Delta(\mathbf{v}_{\gamma}(\lambda))(\mathbf{v}_{\gamma}(\lambda)^{-1}\ot \mathbf{v}_{\gamma}(\lambda)^{-1})\\
&=1-\lambda d_B(v_{\gamma})+\text{higher $Q$-degree terms}
\end{array}
\]
is a twist for $u_q^B$, by Lemma~\ref{lem:vJ}.  It follows that the homogenous degree $\gamma$ term $\lambda d_B(v_{\gamma})$ is in $u_q$ for each $\lambda\in \C$.  Since $v_{\gamma}\notin u_q$, the the cohomology class $[d_B(v_\gamma)]$ is a nonvanishing class of degree $\gamma$ in $H^2((u^\ast_q)_B)$, by Lemma~\ref{lem:dBv}.
\par

Since $\dim\left(H^2((u^\ast_q)_B)_{\gamma}\right)=1$,  nonvanishing of $[d_B(v_{\gamma})]$ implies that there exists $c\in \C$ so that
\[
[d_B(c v_{\gamma})]=c[d_B(v_{\gamma})]=[J_{\gamma}].
\]
It follows that
\[
\mathbf{v}_{\gamma}(c)\cdot J=1+J_d-c d_B(v_{\gamma})+\text{higher $\mathbb{Z}$-degree terms}
\]
is such that $[(\mathbf{v}_{\gamma}(c)\cdot J)_\gamma]=0$.  As was argued above, there now exists a $v'\in \mathbb{V}_{small}$ so that
\[
v'\cdot (\mathbf{v}_{\gamma}(c)\cdot J)=(v'\mathbf{v}_{\gamma}(c))\cdot J
\]
satisfies $\left((v'\mathbf{v}_{\gamma}(c))\cdot J\right)_{d''}=0$ for $0<d''<d$ and $\left((v'\mathbf{v}_{\gamma}(c))\cdot J\right)_d=J_d-J_\gamma$.  If we take in this case $v=v'\mathbf{v}_{\gamma}(c)$, we arrive at the desired result for twists of $u_q^B$.
\par

Let us return now to the case of a twist for $u_q$ of the form $J'=B+\sum_{n> 0}J'_n$, with $B$ a bilinear form on $G^\vee$ and $J'-B$ vanishing in degrees $<d$.  For $J=B^{-1}J'$ we have just seen that there is a unit $v\in \mathbb{V}$ so that gauging $J$ by $v$ produces a twist for $u_q^B$ which vanishes in all degrees $0<d'\leq d$.  By Lemma~\ref{lem:Uset} it follows that $B^{-1}(v\cdot J')$ vanishes in all degrees $0<d\leq d$.  Hence $v\cdot J'$ itself vanishes in degrees $0<d'\leq d$, since $B\in u_q\ot u_q$ is a degree $0$ unit.
\end{proof}

We can now prove the main result.

\begin{proof}[Proof of Theorem~\ref{thm:twists}]
After gauging by an element in $\C[G]$ if necessary, we may write our twist $J$ as $J=B+\sum_{n> 0} J_n$, where $B$ is an alternating form on $G^\vee$.  Take $N$ the maximal $\mathbb{Z}$-degree with $(u_q)_N\neq 0$.  By Lemma~\ref{lem:elimination} we can define recursively a sequence of units $v_1,\dots,v_{N^2}\in \mathbb{V}$ so that
\[
\left((v_iv_{i-1}\dots v_1)\cdot J\right)-B
\]
has first nonvanishing term in a $\mathbb{Z}$-degree $d_i>i$.  Hence, if we define $v_{tot}=(v_{N^2}\dots v_1)$ then $(v_{tot}\cdot J)-B$ vanishes in degrees $\leq N^2$.  Since $v_{tot}\cdot J$ is a twist in $u_q\ot u_q$, by Lemma~\ref{lem:vJ}, and $u_q\ot u_q$ vanishes in degree $>N^2$, we see that $(v_{tot}\cdot J)-B=0$.  Rather, $v_{tot}\cdot J=B$.
\par

Since $\mathbb{V}_{small}$ is normal in $\mathbb{V}$, and $\mathbb{V}$ is generated by $\mathbb{V}_{dp}$ and $\mathbb{V}_{small}$, we may write $v_{tot}=v^{-1}v'$ with $v\in \mathbb{V}_{dp}$ and $v'\in \mathbb{V}_{small}$.  Since $v_{tot}\cdot J=B$ we have then $v'\cdot J=v\cdot B$.  Hence $J$ is gauge equivalent to the twist $v\cdot B$.
\end{proof}

\section{Algebraic group actions and tensor equivalences}
\label{sect:proofs}

In this section we prove Theorems A and B of the introduction.

\subsection{Algebraic groups and twisted automorphisms}
\label{sect:twAuts}
We recall some information from~\cite{arkhipovgaitsgory03,uqJ}.
\par

One takes the graded dual of the exact sequence of Hopf algebras
\[
1\to u_q(\b)\to U_q(\b)\to U(\n)\to 1
\]
to arrive at another exact sequence
\[
1\to \O(\mathbb{U})\to \U^{DK}\to u_q(\b)^\ast\to 1,
\]
where $\mathbb{U}$ the unipotent group corresponding to $\n$.  From such a sequence Arkhipov and Gaitsgory show that there will be an equivalence of tensor categories between the de-equivariantization $\mathrm{corep}(\U^{DK})_\mathbb{U}$ and $\rep(u_q)$~\cite{arkhipovgaitsgory03,agp14}.
\par

The de-equivariantization is a certain monoidal subcategory of $\mathrm{Coh}(\mathbb{U})$, and it inherits an action of $\mathbb{U}$ which is induced by the translation action of $\mathbb{U}$ on $\mathrm{Coh}(\mathbb{U})$.  Whence we get a natural action of $\mathbb{U}$ on $\rep(u_q)$.  Just as in~\cite{uqJ}, one can verify that the $1$-parameter subgroups $\exp(\lambda e_\alpha)\subset \mathbb{U}$, $\alpha\in \Gamma$, act on $\rep(u_q)$ by way of the {\it twisted automorphisms}
\[
(\exp_\alpha^\lambda,J^\lambda_\alpha),
\]
where $\exp_\alpha^\lambda:=\exp(\mathrm{ad}_{\lambda E^{(l)}_\alpha})=\mathrm{Ad}_{\exp(\lambda E^{(l)}_\alpha)}$ and
\[
J^\lambda_\alpha=\Delta\left(\exp(\lambda E^{(l)}_\alpha)\right)\left(\exp(-\lambda E^{(l)}_\alpha)\ot \exp(-\lambda E^{(l)}_\alpha)\right).
\]
\par

Recall that a twisted automorphism $(\phi,J)$ of a Hopf algebra $H$ is a pair consisting of a twist $J$ and a Hopf isomorphism $\phi:H\to H^J$ (see~\cite[Thm. 2.2]{ngschauenburg08} and~\cite{kassel12,davydov10}).  We have already seen in Section~\ref{sect:V} that $\exp_\alpha^\lambda:u_q\to u_q$ is an algebra automorphism and that $J_\alpha^\lambda$ is a twist for $u_q$.
\par

The collection of twisted automorphisms of a Hopf algebra $H$ forms a group under the product
\[
(\phi,J)(\phi',J') = (\phi\phi', J\phi^{\ot 2}(J')),
\]
and this group acts naturally on the left of the set $\mathrm{Tw}(H)$ of gauge equivalence classes of twists by $(\phi,J)\cdot [J']=[J\phi^{\ot 2}(J')]$.\footnote{Our group operations are {\it opposite} that of some other references.  With our conventions, the association $(\phi,J)\mapsto [\phi,J]\in \Aut_{\ot}(\rep(H))$ will be group anti-homomorphism.}

\begin{lemma}\label{lem:twAut}
For any $v\in \mathbb{V}_{dp}$ the pair $(\mathrm{Ad}_v,\Delta(v)(v^{-1}\ot v^{-1}))$ is a twisted automorphism for $u_q$.  Furthermore, the collection of all such pairs 
\[
\mathfrak{V}_{dp}=\{(\mathrm{Ad}_v,\Delta(v)(v^{-1}\ot v^{-1})):v\in\mathbb{V}_{dp}\}
\]
forms a subgroup in the group of twisted automorphisms for $u_q$.
\end{lemma}

\begin{proof}
The fact that each pair $(\mathrm{Ad}_{v},\Delta(v)(v^{-1}\ot v^{-1}))$ is a twisted automorphism follows from Lemmas~\ref{lem:exp} and~\ref{lem:vJ}, and the computation
\[
\Delta^{\Delta(v)(v^{-1}\ot v^{-1})}(\mathrm{Ad}_v(h))=(v\ot v)\Delta(v)^{-1}\Delta(vhv^{-1})\Delta(v)(v^{-1}\ot v^{-1})=\mathrm{Ad}_{v\ot v}(\Delta(h)).
\]
For any two twisted automorphisms of the form $\left(\mathrm{Ad}_{v_1},\Delta(v_1)(v_1^{-1}\ot v_1^{-1})\right)$ and $\left(\mathrm{Ad}_{v_2},\Delta(v_2)(v_2^{-1}\ot v_2^{-1})\right)$, with the $v_i\in\hat{U}_q$, we have
\[
\begin{array}{l}
\left(\mathrm{Ad}_{v_1},\Delta(v_1)(v_1^{-1}\ot v_1^{-1})\right)\cdot\left(\mathrm{Ad}_{v_2},\Delta(v_2)(v_2^{-1}\ot v_2^{-1})\right)\\\\
=\left(\mathrm{Ad}_{v_1v_2},\ \Delta(v_1)(v_1^{-1}\ot v_1^{-1})\mathrm{Ad}_{v_1\ot v_1}\left(\Delta(v_2)(v_2^{-1}\ot v_2^{-1})\right)\right)\\\\
=\left(\mathrm{Ad}_{v_1v_2},\ \Delta(v_1)\Delta(v_2)(v_2^{-1}\ot v_2^{-1})(v_1^{-1}\ot v_1^{-1})\right)\\\\
=\left(\mathrm{Ad}_{v_1v_2},\ \Delta(v_1v_2)((v_1v_2)^{-1}\ot (v_1v_2)^{-1})\right).
\end{array}
\]
Since $\mathbb{V}_{dp}$ is a group, the above calculation also shows that $\mathfrak{V}_{dp}$ is a subgroup in the group of twisted automorphisms.
\end{proof}

\subsection{A classification of twists for $u_q(\b)$ in terms of the $\mathbb{U}$-action}

\begin{theorem}\label{thm:unip}
Under the above natural action of $\mathbb{U}$ on $\mathrm{Tw}(u_q(\b))$, there is an equality $\mathrm{Tw}\left(u_q(\b)\right)=\mathbb{U}\cdot \mathrm{Alt}(G^\vee)$.
\end{theorem}

\begin{proof}
Take $\mathfrak{V}_{dp}$ as in Lemma~\ref{lem:twAut}.  Since the generating $1$-parameter subgroups $\exp(\lambda e_\lambda)$ of $\mathbb{U}$ act on $\mathrm{Tw}(u_q)$ via the generators of $\mathfrak{V}_{dp}$, we have
\[
\mathbb{U}\cdot [J]=\mathfrak{V}_{dp}\cdot [J]
\]
for any class $[J]\in \mathrm{Tw}(u_q)$.  By Theorem~\ref{thm:twists} we have $[J]=[v\cdot B]$ for some $v\in \mathbb{V}_{dp}$ and alternating form $B$ on $G^\vee$.  One checks directly that 
\[
[v\cdot B]=(\mathrm{Ad}_v,\Delta(v)(v^{-1}\ot v^{-1}))\cdot [B]\in \mathfrak{V}_{dp}\cdot [B].
\]
Hence
\[
\mathrm{Tw}(u_q)=\mathfrak{V}_{dp}\cdot \mathrm{Tw}(u_q)=\mathfrak{V}_{dp}\cdot [\mathrm{Alt}(G^\vee)]=\mathbb{U}\cdot [\mathrm{Alt}(G^\vee)].
\]
Form the classification of twists for abelian group algebras~\cite[Cor. 5.8]{guillotkassel09}, and the fact that the Hopf inclusion $\C[G]\to u_q$ is split, one finds that the mapping $\mathrm{Alt}(G^\vee)\to \mathrm{Tw}(u_q)$ is injective.  Hence $[\mathrm{Alt}(G^\vee)]=\mathrm{Alt}(G^\vee)$ and $\mathrm{Tw}(u_q)=\mathbb{U}\cdot \mathrm{Alt}(G^\vee)$.
\end{proof}

\subsection{Tensor equivalence and Hopf algebras}

We use Theorem~\ref{thm:twists} to classify all Hopf algebras which are tensor equivalent to $u_q$, up to Hopf isomorphism.  The following result can alternately be deduced from the works of Andruskiewitsch, Angiono, Iglesias, and Schneider~\cite{ASIII,aai17,angionoiglesias}.

\begin{theorem}[AAIS]\label{thm:equiv}
Suppose $H$ is a Hopf algebra such that $\rep(H)$ and $\rep(u_q(\b))$ are equivalent as tensor categories.  Then there is an alternating form $B$ on $G^\vee$ such that $H\cong u_q(\b)^B$ as Hopf algebras.  When $l$ is coprime to the determinant of the Cartan matrix for $\g$, the form $B$ is uniquely determined by $H$.  In general, if two forms $B$ and $B'$ admit a Hopf isomorphism $u_q(\b)^B\cong u_q(\b)^{B'}$ then their restrictions to $G$, along the Killing map, are equal.
\end{theorem}

\begin{proof}
First, the existence of an invertible $(H,u_q)$-bimodule implies that $H$ is finite dimensional.  We verify the existence of such a form $B$, then address uniqueness.
\par

By Lemma~\ref{lem:twAut}, for each $v\in \mathbb{V}_{dp}$ the conjugation $\mathrm{Ad}_v$ restricts to an algebra automorphism of $u_q$.  Given any twist $J$ for $u_q$, and $v\in \mathbb{V}$, a simple calculation verifies that $\mathrm{Ad}_v$ provides a Hopf isomorphism between the twisted algebras $\mathrm{Ad}_{v}:u_q^{(v\cdot J)}\to u_q^J$.  It follows that any twist $J$ for $u_q$ will admit an alternating form $B$ on $G^\vee$ so that $u_q^J\cong u_q^B$ as a Hopf algebra, by Theorem~\ref{thm:twists}.
\par

Consider now any Hopf algebra $H$ with $\rep(H)\overset{\sim}\to \rep(u_q)$ as tensor categories.  Then there is a Hopf isomorphism between $H$ and $u_q^J$ for some twist $J$, by~\cite[Thm. 2.2]{ngschauenburg08}.  Hence there is a Hopf isomorphism $H\cong u_q^B$ for some alternating $B$.
\par

As for our uniqueness claim, suppose there is a Hopf isomorphism $\psi:u_q^{B'}\overset{\cong}\to u_q^{B}$ for two alternating forms $B'$ and $B$.  Note that $B'$, and $(B')^{-1}$, remain twists for $u_q^{B'}$.  So we may twist both $u_q^{B'}$ and $u_q^{B}$ to arrive at a Hopf isomorphism $u_q\overset{\cong}\to u_q^{B\psi^{\ot 2}(B')^{-1}}$.  Hence it suffices to assume $B'=1$ and show that if $u_q\cong u_q^B$ for an alternating form $B$, then $B$ restricts to the trivial form on $G$.
\par

We may take the dual of such a isomorphism to get a Hopf isomorphism $\phi:u^\ast_q\overset{\cong}\to (u^\ast_q)_B$.  By Lemma~\ref{lem:DKD} $u_q^\ast$, which is a Hopf quotient of the modified De Concini-Kac algebra $\U^{DK}$, is generated by the grouplikes $G^\vee$ and skew primitives $X_\alpha$.  By considering the $J$-adic filtration on $u_q$, with respect to the Jacobson radical $J=(E_\alpha:\alpha\in \Gamma)$, one can also see that the $X_\alpha$ provide all skew primitives in the dual $u^\ast_q$.
\par

Since $\phi$ preserves grouplikes and skew primitives there is a set bijection $\phi':\Gamma\to \Gamma$, and scalars $c_\alpha\in\mathbb{C}^\times$ for each $\alpha\in \Gamma$, so that 
\[
\phi(X_\alpha)= c_\alpha X_{\phi'(\alpha)}.
\]
Compatibility with the comultiplication implies $\phi(\alpha)\in G^\vee$ is the image of $\phi'(\alpha)\in Q/lQ$ under the Killing map.  After precomposing with a scaling on the generators of $u^+_q$, we may assume that each $c_\alpha=1$.
\par

We have now
\[
\begin{array}{rl}
q^{(\alpha,\beta)-(\phi'(\alpha),\phi'(\beta))}\phi(X_\alpha)& = q^{-(\phi'(\alpha),\phi'(\beta))}\phi(\beta X_\alpha \beta^{-1})\\
&=q^{-(\phi'(\alpha),\phi'(\beta))}\phi(\beta)\cdot_B X_{\phi'(\alpha)} \cdot_B \phi(\beta)^{-1}\\
&=\frac{B(\phi(\alpha),\phi(\beta))}{B(\phi(\beta),\phi(\alpha))}X_{\phi'(\alpha)}=B^2(\phi(\alpha),\phi(\beta))X_{\phi'(\alpha)}
\end{array}
\]
for each $\alpha,\beta\in \Gamma$ (see Lemma~\ref{lem:relations}).  This implies
\[
B^2(\phi(\alpha),\phi(\beta))=q^{(\alpha,\beta)-(\phi'(\alpha),\phi'(\beta))}
\]
for each $\alpha,\beta\in\Gamma$.  Since the right hand side of the above equation is symmetric, we conclude that $B^2$ restricted to the image of $G$ in $G^\vee$ is symmetric.  Since $2\nmid l$ we can take the square root of $B^2$ to find that the restriction of $B$ to $G$ is also symmetric.  This forces the restriction of $B$ to be $1$, as desired.
\par

When $l$ is coprime to the determinant of the Cartan matrix the killing map $\kappa:G\to G^\vee$ is an isomorphism.  So uniqueness follows.
\end{proof}

\section{Proofs of Propositions~\ref{prop:HUB} and~\ref{prop:HuB}}
\label{sect:UDK}

We prove Propositions~\ref{prop:HUB} and~\ref{prop:HuB}, which claim equalities
\[
H^1(\U^{DK}_B)=H^2(\U^{DK}_B)=0,
\]
\[
H^1((u^\ast_q)_B)=0\ \ \text{and}\ \ H^2((u^\ast_q)_B)=\n^{(l)}
\]
respectively.  Our computations are based on explicit presentations of the algebras $\U^{DK}_B$ and $(u^\ast_q)_B$, and the majority of the material below is dedicated to exhibiting such presentations.
\par

{\it We fix for the section a bilinear form $B$ on $G^\vee$}.  Take $u_q=u_q(\b)$ and $\U^{DK}$ as in Section~\ref{sect:modDK}.

\subsection{A presentation of the algerba $\U^{DK}_B$}

The left and right $\C[G]$-actions on $\U^{DK}$ are given by
\[
g\cdot \omega=\omega\cdot g=g(\omega)\omega,\ \ g\cdot X_\alpha=g(\alpha^{-1})X_\alpha,\ \ \mathrm{and}\ \ X_\alpha\cdot g=X_\alpha,
\]
for $g\in G$ and $\omega\in G^\vee$.  Hence, for the multiplication $\cdot_B$ of the cocycle twist $\U^{DK}_B$, we have
\[
X_\alpha\cdot_B X_\beta=B(\alpha^{-1},\beta^{-1})X_\alpha X_\beta=B(\alpha,\beta)X_\alpha X_\beta,
\]
\[
\omega\cdot_B X_\alpha=B(\omega,\alpha)^{-1} \omega X_\alpha,
\]
and
\[
\omega\cdot_B\chi=\frac{B(\omega,\chi)}{B(\omega,\chi)}\omega\chi=\omega\chi
\]
for each $\omega,\chi\in G^\vee$.  One establishes by induction

\begin{lemma}\label{lem:496}
For any collection of generators $X_{\alpha_1},\dots, X_{\alpha_m}$ of $\U^{DK}$, and $\omega\in G^\vee$, we have
\[
\omega\cdot_B X_{\alpha_1}\cdot_B\dots\cdot_B X_{\alpha_m}=B^{-1}(\omega,\alpha_1\dots\alpha_m)(\prod_{1\leq i<j\leq m}B(\alpha_i,\alpha_j))\omega X_{\alpha_1}\dots X_{\alpha_m}.
\]
In particular, $\U^{DK}_B$ is generated as an algebra by $\C[G^\vee]$ and the $X_\alpha$.
\end{lemma}

\begin{lemma}\label{lem:relations}
The generators of the cocycle twist $\U^{DK}_B$ satisfy the commutator relations
\begin{equation}\label{eq:com}
\omega\cdot_B X_\alpha\cdot_B \omega^{-1}=\omega(K_\alpha)\frac{B(\alpha,\omega)}{B(\omega,\alpha)}X_\alpha
\end{equation}
and the $B$-Serre relations
\begin{equation}\label{eq:Bserre}
\sum_{i=0}^{1-\langle\alpha,\beta\rangle}(-1)^i\frac{B(\alpha,\alpha)^{\frac{\langle\alpha,\beta\rangle(1-\langle\alpha,\beta\rangle)}{2}}}{B(\alpha,\beta)^{1-\langle\alpha,\beta\rangle-i}B(\beta,\alpha)^i}\qbinom{1-\langle\alpha,\beta\rangle}{i}_{q^{-1}_\alpha}X_\alpha^{1-\langle\alpha,\beta\rangle-i}\cdot_BX_\beta\cdot_BX_\alpha^i=0.
\end{equation}
\end{lemma}

\begin{proof}
The commutator relation follows by Lemma~\ref{lem:PD2} and the calculation
\[
\omega\cdot_B X_\alpha\cdot_B \omega^{-1}=\frac{B(\alpha,\omega)}{B(\omega,\alpha)}\omega X_\alpha\omega^{-1}.
\]
The $B$-Serre relations follow from the the previous lemma and the relations for $\U^{DK}$.
\end{proof}

\begin{proposition}\label{prop:presUB}
The algebra projection
\[
\pi:\frac{\C\langle \bar{\omega},\bar{X}_\alpha:\omega\in G^\vee,\alpha\in \Gamma\rangle}{(1_{G^\vee}-1,\ \bar{\omega}\bar{\chi}-\overline{\omega\chi},\ \mathrm{relations}\ \eqref{eq:com}\ \mathrm{and}\ \eqref{eq:Bserre})}\to \U^{DK}_B,\ \left\{\begin{array}{l}
\bar{\omega}\mapsto \omega\\
\bar{X}_\alpha\mapsto X_\alpha
\end{array}\right.
\]
is an isomorphism.
\end{proposition}

We adopt an unorthodox notation for the proof. For a Hopf algebra $H$, twist $J\in H\ot H$, and $H$-module algebra $R$, we can twist the multiplicaiton on $R$ to arrive at a new algebra with multiplication $r\cdot_{new}r'=\mathrm{mult}(J\cdot (r\ot r'))$.  We write this new algebra as $R(J)$, to not confuse with the cocycle or Drinfeld twist of a Hopf algebra.

\begin{proof}
Let $R$ denote the domain of $\pi$.  We grade $R$ by the root lattice $Q$ by taking $\deg_Q(\bar{\omega})=0$ and $\deg_Q(\bar{X}_\alpha)=-\alpha$.  Under this grading $\pi$ is homogenous.  Hence to prove that $\pi$ is an isomorphism it suffices to prove that $R$ and $\U^{DK}_B$ have the same dimension in each $Q$-degree.
\par

We can give $R$ the $G\times G^{op}$-action defined by $(g,h)\cdot \omega=g(\omega)h^{-1}(\omega)\bar{\omega}$ and $(g,h)\cdot \bar{X}_\alpha=g(\alpha^{-1})\bar{X}_\alpha$.  Under this action $R$ becomes a $\C[G\times G^{op}]$-module algebra, and $\pi$ is a map of $\C[G\times G^{op}]$-module algebras.
\par

Let $\mathbb{B}\in \C[G\times G^{op}]^{\ot 2}$ be the bilinear form corresponding to $B\ot B\in\C[G]^{\ot 2}\ot \C[G^{op}]^{\ot 2}$ under the obvious isomorphism $\C[G]^{\ot 2}\ot \C[G^{op}]^{\ot 2}\cong \C[G\times G^{op}]^{\ot 2}$.  We twist $\pi$ to arrive at a surjection
\[
\pi(\mathbb{B}^{-1}):R(\mathbb{B}^{-1})\to (\U^{DK}_B)(\mathbb{B}^{-1}).
\]
(Note that $\pi(\mathbb{B})=\pi$ as a vector space map.)  The algebra $R(\mathbb{B}^{-1})$ now has the relations of the De Concini Kac algebra.  So, for each $\mu\in Q$,
\[
\dim(R_\mu)=\dim(R(\mathbb{B}^{-1})_\mu)\leq \dim((\U^{DK})_\mu)=\dim((\U^{DK}_B)_\mu).
\]
The inverse inequality follows by surjectivity or $\pi$, and hence $\dim(R_\mu)=\dim((\U^{DK}_B)_\mu)$ for each $\mu\in Q$.  We conclude that $\pi$ is an isomorphism.
\end{proof}

Let $\U^-_B$ denote the subalgebra in $\U^{DK}_B$ generated by the $\{X_\alpha:\alpha\in \Gamma\}$.  Note that $\U^-_B$ is a $Q$-graded subalgebra in $U^{DK}_B$.

\begin{proposition}\label{prop:presUB+}
The map
\[
\pi^-:\frac{\C\langle\bar{X}_\alpha:\alpha\in \Gamma\rangle}{(\mathrm{relations}\ \eqref{eq:Bserre})}\to \U^-_B,\ \ \bar{X}_\alpha\mapsto X_\alpha,
\]
is an algebra isomorphism.  Furthermore, if we let $G^\vee$ act on $\U^-_B$ by conjugation in $\U^{DK}_B$, multiplication provides an isomorphism $\U^-_B\rtimes \C[G^\vee]\cong\U^{DK}_B$.
\end{proposition}

\begin{proof}
Surjectivity of the multiplication map $\U^-_B\rtimes \C[G^\vee]\to \U^{DK}_B$ is clear, since $\U^{DK}_B$ is generated by $G^\vee$ and the $X_\alpha$.  Injectivity follows by Lemma~\ref{lem:496} and a dimension count in each homogenous $Q$-degree.  As for the map $\pi^-$, let us denote the domain by $R^-$.  We let $\C[G^\vee]$ act on $R^-$ by ``conjugation"
\[
\omega\cdot \bar{X}_\alpha=\omega(K_\alpha)\frac{B(\alpha,\omega)}{B(\omega,\alpha)}\bar{X}_\alpha.
\]
Take the smash product to get a surjective map
\[
\pi^-\rtimes \C[G^\vee]: R^-\rtimes \C[G^\vee]\to \U^-_B\rtimes \C[G^\vee].
\]
The algebra $R^-\rtimes \C[G^\vee]$ satisfies all of the relations of the domain $R$ of $\pi$, from Proposition~\ref{prop:presUB}.  So, just as in Proposition~\ref{prop:presUB}, we find that $\pi^-\rtimes\C[G^\vee]$ is an isomorphism.  Since $\C[G^\vee]$ is faithfully flat over $k$, it follows that $\pi^-$ is an isomorphism.
\end{proof}

\subsection{Proof of Proposition~\ref{prop:HUB}}
\label{sect:HUB}

Before presenting the proof we give some background information.  Since $\U^{DK}_B=\U^-_B\rtimes \C[G^\vee]$, for any $\U^{DK}_B$-modules $M$ and $N$ we have
\[
\Hom_{\U^{DK}_B}(M,N)=\Hom_{\U^-_B}(M,N)^{G^\vee}.
\]
Hence we have an equality of cohomologies $H^\bt(\U^{DK}_B)=H^\bt(\U^-_B)^{G^\vee}$.
\par

We are particularly interested in the case in which $M$ and $N$ are $G$-graded, and the corresponding $G^\vee\subset U^{DK}_B$-action is as follows:
\[
\omega\cdot m=\omega(\deg_G m)\frac{B(\deg_G m,\omega)}{B(\omega,\deg_G m)}m,
\]
for homogenous $m$. (We use the Killing map to identify the $G$-degree with a character in $G^\vee$.) In this case we can decompose the set of morphisms
\[
\Hom_{\U^-_B}(M,N)=\oplus_{g\in G}\Hom_{\U^-_B}(M,N)_g, \ \ \Hom_{\U^-_B}(M,N)_g=\left\{\begin{array}{c}
\text{homogenous degree }g\\
\text{maps }f:M\to N
\end{array}\right\}.
\]
Under this decomposition $G^\vee$ will act on each homogenous $f\in \Hom_{\U^-_B}(M,N)_g$ as
\[
\omega\cdot f=\omega(g)\frac{B(\omega,g)}{B(g,\omega)}f.
\]
Whence we have an easy basis of eigenvectors with respect to which we can calculate the $G^\vee$-invariants.
\par

Note that $\U^-_B$, with $G$-grading induced by its $Q$-grading, admits such a $\U^{DK}_B$-module structure.  As do the shifts $\U^-_B\ot \C_g$ by $g\in G$.

\begin{proof}[Proof of Proposition~\ref{prop:HUB}]
Via the identification $Q/lQ=G$ we will denote $G$-degrees by elements in $Q/lQ$.  We let $\bar{X}$ denote formal variables.  As $\U^-_B$ is a connected $\mathbb{Z}$-graded algebra, we have the minimal $\mathbb{Z}$-graded resolution
\[
P=\cdots\to \U^-_B\ot V\to \U^-_B\ot \C\{\bar{X}_\alpha:\alpha\in\Gamma\}\to \U^-_B\to 0
\]
of the trivial module, where $V$ is the graded vector space of relations
\[
\C\left\{{\Small\sum_{i=0}^{1-\langle\alpha,\beta\rangle}(-1)^i\frac{B(\alpha,\alpha)^{\frac{\langle\alpha,\beta\rangle(1-\langle\alpha,\beta\rangle)}{2}}}{B(\alpha,\beta)^{1-\langle\alpha,\beta\rangle-i}B(\beta,\alpha)^i}\qbinom{1-\langle\alpha,\beta\rangle}{i}_{\zeta_\alpha}\bar{X}_\alpha^{1-\langle\alpha,\beta\rangle-i}\cdot_B\bar{X}_\beta\cdot_B\bar{X}_\alpha^i}\right\}_{\alpha,\beta}.
\]
(See for example~\cite[Sect. 2]{atv07}.)  The resolution $P$ is also $G$-graded, where the $\bar{X}_\alpha$ are of degree $\alpha^{-1}$ and the $B$-Serre relations are of respective degrees $\alpha^{\langle\alpha,\beta\rangle-1}\beta^{-1}$.
\par

Since $P$ is minimal, the first and second cohomologies are given by the duals $H^1(\U^-_B)=(\C\{\bar{X}_\alpha:\alpha\in\Gamma\})^\ast$ and $H^2(\U^-_B)=V^\ast$, as a $G$-graded spaces.  We have $\alpha\cdot \bar{X}_\alpha^\ast=q^2 \bar{X}^\ast_\alpha$ for each dual function $\bar{X}_\alpha^\ast$ to the $\bar{X}_\alpha$, and hence $H^1(\U^{DK}_B)=H^1(\U^-_B)^{G^\vee}=0$.  We claim also $H^2(\U^{DK}_B)=0$.
\par

The space $H^2(\U^-_B)=V^\ast$ is spanned by functions $f$ of $G$-degrees $\alpha^{1-\langle\alpha,\beta\rangle}\beta$.  Since the fraction $\frac{B(\omega,?)}{B(?,\omega)}$ is alternating, we can act on such a function by its degree to get
\begin{equation}\label{eq:703}
(\alpha^{1-\langle\alpha,\beta\rangle}\beta)\cdot f=q^{(1-\langle\alpha,\beta\rangle)^2(\alpha,\alpha)+2(1-\langle\alpha,\beta\rangle)(\alpha,\beta)+(\beta,\beta)}f.
\end{equation}
By considering the possible lengths of roots in $\Gamma$ we have
\[
(1-\langle\alpha,\beta\rangle)^2(\alpha,\alpha)+2(1-\langle\alpha,\beta\rangle)(\alpha,\beta)+(\beta,\beta)
\]
\[
=\left\{
\begin{array}{ll}
4\ \mathrm{or}\ 6 &\text{when $\g$ is simply laced}\\
4,\ 6,\ \mathrm{or}\ 10 &\text{in types $B_n$ or $C_n$}\\
4,\ 6,\ 10,\ \mathrm{or}\ 12 &\text{in type }F_4\\
14 &\text{in type }G_2.
\end{array}\right.
\]
Since $l$ is odd, the restrictions on $l$ introduced in Section~\ref{sect:uq} are sufficient to ensure that the order of $q$ avoids each of the above integers.  So the coefficient appearing in~\eqref{eq:703} are never $1$.  It follows that the invariants vanish, and $H^2(\U^{DK}_B)=H^2(\U^-_B)^{G^\vee}=0$.
\end{proof}

\subsection{Preliminary information for the proof of Proposition~\ref{prop:HuB}}

We now consider the small algebra $u_q=u_q(\b)$.  If we take the (graded) dual of the $Q$-graded inclusion $u_q\to U_q$ we get a $Q$-graded Hopf projection $\U^{DK}\cong U_q^\#\to u_q^\ast$.  Since $\U^{DK}$ is generated by the functions $X_\alpha$, and $G^\vee$, $u_q^\ast$ will be generated by these functions as well.  Let $u^-$ denote the subalgebra of $u_q^\ast$ generated by the $X_\alpha$.

\begin{lemma}\label{lem:756}
The projection $\U^{DK}\to u_q^\ast$ gives $u_q^\ast$ as the quotient
\[
u_q^\ast=\U^{DK}/(X_\mu^l:\mu\in\Phi^+)=u^-\rtimes \C[G^\vee].
\]
\end{lemma}

\begin{proof}
Note that $u_q^\ast$ is pointed, coradically graded, and generated by the skew primitives $X_\alpha$.  Hence $u_q^\ast$ is a bosonization $\mathfrak{B}(V)\rtimes \C[G^\vee]$, where $V$ is the Yetter-Drinfeld $G^\vee$-module spanned by the skew primitives $X_\alpha$ in $u_q^\ast$ and $\mathfrak{B}(V)$ is the corresponding Nichols algebra~\cite{nichols78,andruskiewitschschneider00}.
\par

Consider alternatively the span $V'$ of the skew-primitives in $u_{q^{-1}}(\b_-)$.  Since the algebra $\mathfrak{B}(V)$ only depends on the braiding for $V$, the relations of Lemmas~\ref{lem:PD1} and~\ref{lem:PD2} imply an algebra isomorphism $u_{q^{-1}}^-=\mathfrak{B}(V')\overset{\cong}\to \mathfrak{B}(V)$, $F_\alpha\mapsto X_\alpha$.  It follows that the $X_\alpha$ satisfy the same relations as the $F_\alpha$ in $u_{q^{-1}}^-$.  In particular $X_{\mu}^l=0$ for all $\mu\in \Phi^+$.  Whence we have a surjection $\U^{DK}/(X^l_\mu)\to u_q^\ast$.  Agreement of dimensions implies that the surjection is an isomorphism.
\end{proof}

Following our notation for $u^\ast_q$, we write $u^-_B$ for the subalgebra of the cocycle twist $(u^\ast_q)_B$ generated by the $X_\alpha$. 

\begin{corollary}
For any bilinear form $B\in \C[G]\ot\C[G]$ we have a minimal presentations
\[
u^-_B=\frac{\C\langle \bar{X}_\alpha:\alpha\in \Gamma\rangle}{(\mathrm{relations}\ \eqref{eq:Bserre},\ \bar{X}_\mu^l:\ \mu\in \Phi^+)}.
\]
\end{corollary}

By a {\it minimal} presentation we mean that the given relations are linearly independent in the quotient $I/I^2$, where $I$ is the kernel of the projection from the free algebra $\C\langle\bar{\omega},\bar{X}_\alpha:\omega,\alpha\rangle \to (u_q^\ast)_B$.

\begin{proof}
One can use Lemma~\ref{lem:496} to check that $X_\mu^{\cdot_Bl}=X_\mu^l$ in $\U^{DK}_B$.  So there is no ambiguity in writing $X_\mu^l$ here.  The presentations now follow from the presentation of $\U^-_B$ given at Propositions~\ref{prop:presUB+}, and Lemma~\ref{lem:756}.  Minimality can be argued via centrality of the $X_\mu^l$ and by considering Lusztig's basis of $\U^-_B$~\cite{deconcinikac89,lusztig90II}.
\end{proof}

\subsection{Proof of Proposition~\ref{prop:HuB}}

As was mentioned in Section~\ref{sect:HuB}, the cohomology $H^\bt((u^\ast_q)_B)$ is $Q$-graded.  This was proposed to follow from the fact that the standard complex $C^\bt((u^\ast_q)_B)$ is $Q$-graded, which is true.  We can, however, also find this $Q$-grading by identifying the cohomology of $(u^\ast_q)_B$ with the derived $Q$-graded Hom functor on the category of $Q$-graded $(u_q^\ast)_B$-modules.

\begin{proof}[Proof of Proposition~\ref{prop:HuB}]
We proceed as in the proof of Theorem~\ref{prop:HUB}.  We want to first give the cohomology of $u^-_B$, then take the invariants to get $H^\bt(u^-_B)^{G^\vee}=H^\bt(u_B^\ast)$.
\par

We have the minimal $Q$-graded resolution
\[
\cdots\to (u^-_B\ot W)\oplus (u^-_B\ot V)\to u^-_B\ot \C\{\bar{X}_\alpha:\alpha\in\Gamma\}\to u^-_B\to 0
\]
of $k$ over $u^-_B$, where $V$ is spanned by the $B$-Serre relations and $W=\C\{\bar{X}_\mu^l:\mu\in \Phi^+\}$. Note that $W$ lay in $Q$-degrees $-l\Phi^+$.  (See~\cite[Sect. 2]{atv07}.)
\par

Hence 
\[
H^1(u^-_B)=\C\{\bar{X}_\alpha:\alpha\in\Gamma\}^\ast,\ \ \mathrm{and}\ \ H^2(u^-_B)=W^\ast\oplus V^\ast.
\]
As we saw in the proof of Proposition~\ref{prop:HUB}, $H^1(u^-_B)^{G^\vee}=0$ and $(V^\ast)^{G^\vee}=0$.  The subspace $W^\ast$ is $G^\vee$-invariant, since each $\omega\in G^\vee$ acts trivially on elements of $Q$-degrees $lQ$.  This gives
\[
H^1\left((u_q^\ast)_B\right)=H^1(u^-_B)^{G^\vee}=0\ \ \mathrm{and}\ \ H^2\left((u_q^\ast)_B\right)=H^2(u^-_B)^{G^\vee}=W^\ast.
\]
We identify $\n^{(l)}$ with $W^\ast$ by taking each $e_\mu$ to the function $\bar{X}_\nu^l\mapsto \delta_{\mu,\nu}$.
\end{proof}

\bibliographystyle{abbrv}

\begin{thebibliography}{10}

\bibitem{aai17}
N.~Andruskiewitsch, I.~Angiono, and A.~G. Iglesias.
\newblock Liftings of nichols algebras of diagonal type {I}. cartan type {A}.
\newblock {\em International Mathematics Research Notices}, 2017(9):2793--2884,
  2017.

\bibitem{andruskiewitschschneider00}
N.~Andruskiewitsch and H.-J. Schneider.
\newblock Finite quantum groups and {C}artan matrices.
\newblock {\em Adv. Math.}, 154(1):1--45, 2000.

\bibitem{ASIII}
N.~Andruskiewitsch and H.-J. Schneider.
\newblock On the classification of finite-dimensional pointed {H}opf algebras.
\newblock {\em Ann. of Math.}, 171(1):375--417, 2010.

\bibitem{agp14}
I.~Angiono, C.~Galindo, and M.~Pereira.
\newblock De-equivariantization of {H}opf algebras.
\newblock {\em Algebr. Represent. Theory}, pages 1--20, 2014.

\bibitem{angionoiglesias}
I.~Angiono and A.~G. Iglesias.
\newblock Liftings of {N}ichols algebras of diagonal type {II}. {A}ll liftings
  are cocycle deformations.
\newblock {\em preprint \href{http://arxiv.org/abs/1605.03113}{\tt
  arXiv:1605.03113}}.

\bibitem{arkhipovgaitsgory03}
S.~Arkhipov and D.~Gaitsgory.
\newblock Another realization of the category of modules over the small quantum
  group.
\newblock {\em Adv. Math.}, 173(1):114--143, 2003.

\bibitem{atv07}
M.~Artin, J.~Tate, and M.~Van~den Bergh.
\newblock Some algebras associated to automorphisms of elliptic curves.
\newblock {\em The Grothendieck Festschrift}, pages 33--85, 2007.

\bibitem{belavindrinfeld}
A.~Belavin and V.~Drinfeld.
\newblock {\em Triangle equations and simple {L}ie algebras}.
\newblock Hardwood Academic Publishers, 1998.

\bibitem{davydov10}
A.~Davydov.
\newblock Twisted automorphisms of {H}opf algebras.
\newblock In {\em Noncommutative structures in mathematics and physics}, pages
  103--130. K. Vlaam. Acad. Belgie Wet. Kunsten (KVAB), Brussels, 2010.

\bibitem{den}
A.~Davydov, P.~Etingof, and D.~Nikshych.
\newblock Autoequivalences of tensor categories attached to quantum groups at
  roots of $1$.
\newblock {\em preprint \href{http://arxiv.org/abs/1703.06543}{\tt
  arXiv:1703.06543}}.

\bibitem{deconcinikac89}
C.~De~Concini and V.~G. Kac.
\newblock Representations of quantum groups at roots of 1.
\newblock In {\em Operator algebras, unitary representations, enveloping
  algebras, and invariant theory ({P}aris, 1989)}, volume~92 of {\em Progr.
  Math.}, pages 471--506. Birkh\"auser Boston, Boston, MA, 1990.

\bibitem{egno15}
P.~Etingof, S.~Gelaki, D.~Nikshych, and V.~Ostrik.
\newblock {\em Tensor categories}, volume 205.
\newblock American Mathematical Society, 2015.

\bibitem{etingofnikshych01}
P.~Etingof and D.~Nikshych.
\newblock Dynamical quantum groups at roots of 1.
\newblock {\em Duke Math. J.}, 108(1):135--168, 2001.

\bibitem{etingofostrik04}
P.~Etingof and V.~Ostrik.
\newblock Finite tensor categories.
\newblock {\em Mosc. Math. J.}, 4(3):627--654, 2004.

\bibitem{gelaki02}
S.~Gelaki.
\newblock Semisimple triangular hopf algebras and tannakian categories,
  arithmetic fundamental groups and noncommutative algebra (berkeley, ca,
  1999), 497--515.
\newblock In {\em Proc. Sympos. Pure Math}, volume~70, pages 497--515, 2002.

\bibitem{ginzburgkumar93}
V.~Ginzburg and S.~Kumar.
\newblock Cohomology of quantum groups at roots of unity.
\newblock {\em Duke Math. J}, 69(1):179--198, 1993.

\bibitem{grunenfeldermastnak}
L.~Grunenfelder and M.~Mastnak.
\newblock Pointed {H}opf algebras as cocycle deformations.
\newblock {\em preprint \href{http://arxiv.org/abs/1010.4976}{\tt
  arXiv:1010.4976}}.

\bibitem{guillotkassel09}
P.~Guillot and C.~Kassel.
\newblock Cohomology of invariant {D}rinfeld twists on group algebras.
\newblock {\em Int. Math. Res. Not.}, 2010(10):1894--1939, 2009.

\bibitem{kassel12}
C.~Kassel.
\newblock {\em Quantum groups}, volume 155.
\newblock Springer Science \& Business Media, 2012.

\bibitem{lusztig89}
G.~Lusztig.
\newblock Modular representations and quantum groups.
\newblock {\em Contemp. Math}, 82(1080):59--78, 1989.

\bibitem{lusztig90}
G.~Lusztig.
\newblock Finite dimensional {H}opf algebras arising from quantized universal
  enveloping algebras.
\newblock {\em J. Amer. Math. Soc.}, 3(1):257--296, 1990.

\bibitem{lusztig90II}
G.~Lusztig.
\newblock Quantum groups at roots of 1.
\newblock {\em Geom. Dedicata}, 35(1):89--113, 1990.

\bibitem{mombelli10}
M.~Mombelli.
\newblock Module categories over pointed {H}opf algebras.
\newblock {\em Math. Z.}, 266(2):319--344, 2010.

\bibitem{montgomery}
S.~Montgomery.
\newblock {\em Hopf algebras and their actions on rings}, volume~82 of {\em
  CBMS Regional Conference Series in Mathematics}.
\newblock Published for the Conference Board of the Mathematical Sciences,
  Washington, DC, 1993.

\bibitem{montgomery04}
S.~Montgomery.
\newblock Algebra properties invariant under twisting.
\newblock {\em Hopf Algebras in Noncommutative Geometry and Physics. Lecture
  Notes in Pure and Appl. Math}, 239:229--243, 2004.

\bibitem{uqJ}
C.~Negron.
\newblock Small quantum groups associated to {B}elavin-{D}rinfeld triples.
\newblock {\em preprint \href{http://arxiv.org/abs/1701.00283}{\tt
  arXiv:1701.00283}}.

\bibitem{ngschauenburg08}
S.-H. Ng and P.~Schauenburg.
\newblock Central invariants and higher indicators for semisimple quasi-{H}opf
  algebras.
\newblock {\em Trans. Amer. Math. Soc.}, 360(4):1839--1860, 2008.

\bibitem{nichols78}
W.~D. Nichols.
\newblock Bialgebras of type one.
\newblock {\em Comm. Algebra}, 6(15):1521--1552, 1978.

\bibitem{radford93}
D.~E. Radford.
\newblock Minimal quasitriangular {H}opf algebras.
\newblock {\em J. Algebra}, 157(2):285--315, 1993.

\bibitem{radford11}
D.~E. Radford.
\newblock {\em Hopf Algebras}.
\newblock World Scientific, 2011.

\bibitem{turaev}
V.~G. Turaev.
\newblock {\em Quantum invariants of knots and 3-manifolds}, volume~18 of {\em
  de Gruyter Studies in Mathematics}.
\newblock Walter de Gruyter \& Co., Berlin, 1994.

\end{thebibliography}

\end{document}